\documentclass{amsart}
\usepackage{amsthm,amsmath,amssymb,amscd,amsfonts}

\input xy
\xyoption{all}

\newcommand{\geom}[2]{{#1}\otimes_{#2}\overline{k}}

\DeclareMathOperator{\GL}{GL}

\DeclareMathOperator{\Tr}{Tr}

\DeclareMathOperator{\Hom}{Hom}

\DeclareMathOperator{\Tor}{Tor}

\DeclareMathOperator{\Aut}{Aut}
\DeclareMathOperator{\Sym}{Sym}
\DeclareMathOperator{\Res}{Res}
\DeclareMathOperator{\Pic}{Pic}
\DeclareMathOperator{\Spec}{Spec}
\DeclareMathOperator{\Frac}{Frac}

\DeclareMathOperator{\id}{id}
\DeclareMathOperator{\ev}{ev}

\DeclareMathOperator{\codim}{codim}
\DeclareMathOperator{\coker}{coker}
\DeclareMathOperator{\Frob}{Frob}
\DeclareMathOperator{\leng}{leng}
\DeclareMathOperator{\vol}{vol}
\DeclareMathOperator{\val}{val}
\DeclareMathOperator{\Gal}{Gal}

\def\AA{\mathbb{A}}

\def\FF{\mathbb{F}}
\def\GG{\mathbb{G}}

\def\OO{\mathbb{O}}
\def\PP{\mathbb{P}}

\def\ZZ{\mathbb{Z}}

\def\bR{\mathbf{R}}
\def\bM{\mathbf{M}}
\def\bN{\mathbf{N}}

\def\calA{\mathcal{A}}
\def\calB{\mathcal{B}}

\def\calE{\mathcal{E}}
\def\calF{\mathcal{F}}
\def\calG{\mathcal{G}}
\def\calK{\mathcal{K}}
\def\calL{\mathcal{L}}
\def\calM{\mathcal{M}}
\def\calN{\mathcal{N}}
\def\calO{\mathcal{O}}
\def\calP{\mathcal{P}}
\def\calQ{\mathcal{Q}}

\def\tilY{\widetilde{Y}}
\def\tilU{\widetilde{U}}
\def\tili{\widetilde{\iota}}
\def\tilR{\widetilde{R}}
\def\tila{\widetilde{a}}
\def\tilb{\widetilde{b}}

\def\frgl{\mathfrak{gl}}
\def\frh{\mathfrak{h}}
\def\frs{\mathfrak{s}}
\def\fru{\mathfrak{u}}
\def\frN{\mathfrak{N}}

\def\Ug{\textup{U}}
\def\Sg{\textup{S}}
\def\Hg{\textup{H}}
\def\Ql{\overline{\mathbb{Q}}_{\ell}}
\def\Nm{\textup{Nm}}
\def\one{\mathbf{1}}
\def\OI{\mathbf{O}}
\def\Mloc{\calM^{\textup{loc}}}

\def\Nloc{\calN^{\textup{loc}}}

\def\sMloc{M^{\textup{loc}}}
\def\sNloc{N^{\textup{loc}}}
\def\isom{\stackrel{\sim}{\to}}
\def\Quot{\textup{Quot}}
\def\Hilb{\textup{Hilb}}
\def\char{\textup{char}}
\def\Tot{\textup{Tot}}
\def\cPic{\overline{\calP\textup{ic}}}

\def\Ah{\calA^{\textup{int}}}
\def\Ared{\calA^{\heartsuit}}
\def\Adel{\calA^{\leq\delta}}
\def\Bh{\calB^{\times}}

\def\Mh{\calM^{\textup{int}}}
\def\bMh{\overline{\calM}^{\textup{int}}}
\def\bNh{\overline{\calN}^{\textup{int}}}
\def\Nh{\calN^{\textup{int}}}
\def\MHit{\calM^{\textup{Hit}}}
\def\NHit{\calN^{\textup{Hit}}}
\def\Mdel{\calM^{\leq\delta}}
\def\Ndel{\calN^{\leq\delta}}
\def\Asm{\calA^{\textup{sm}}}
\def\Msm{\calM^{\textup{sm}}}
\def\Nsm{\calN^{\textup{sm}}}
\def\Minf{\calM^{\infty}}
\def\Ninf{\calN^{\infty}}
\def\Ainf{\calA^{\infty}}
\def\AJ{\textup{AJ}}
\def\fh{f^{\textup{int}}}
\def\gh{g^{\textup{int}}}
\def\fsm{f^{\textup{sm}}}
\def\gsm{g^{\textup{sm}}}
\def\fdel{f^{\leq\delta}}
\def\gdel{g^{\leq\delta}}
\def\finf{f^{\infty}}
\def\ginf{g^{\infty}}
\def\finfdel{f^{\infty,\leq\delta}}
\def\ginfdel{g^{\infty,\leq\delta}}
\def\jsm{j^{\textup{sm}}}
\def\div{\textup{div}}
\def\red{\textup{red}}
\def\ellp{\textup{ell}}
\def\Disc{\textup{Disc}}
\def\Art{\textup{Art}}
\def\abel{\textup{ab}}

\swapnumbers
\theoremstyle{plain}
\newtheorem{theorem}[subsubsection]{Theorem}
\newtheorem{lemma}[subsubsection]{Lemma}
\newtheorem{cor}[subsubsection]{Corollary}
\newtheorem{conj}[subsubsection]{Conjecture}
\newtheorem{prop}[subsubsection]{Proposition}
\newtheorem*{claim}{Claim}
\newtheorem*{mthloc}{Main Theorem (Local Part)}
\newtheorem*{mthglob}{Main Theorem (Global Part)}

\theoremstyle{definition}
\newtheorem{defn}[subsubsection]{Definition}

\theoremstyle{remark}
\newtheorem{remark}[subsubsection]{Remark}

\numberwithin{equation}{subsection}

\title[Fundamental Lemma of Jacquet-Rallis]{The fundamental lemma of Jacquet-Rallis in positive characteristics}
\author{Zhiwei Yun}
\address{Current Address: Institute for Advanced Study}
\email{zyun@math.ias.edu}
\date{December 2008; Revise October 2009}

\begin{document}

\begin{abstract}
We prove both the group version and the Lie algebra version of the Fundamental Lemma appearing in a relative trace formula of Jacquet-Rallis in the function field case when the characteristic is greater than the rank of the relevant groups.
\end{abstract}

\maketitle

\section{Introduction}

\subsection{The conjecture of Jacquet-Rallis and its variant}

In the paper \cite{JR}, Jacquet and Rallis proposed an approach to the Gross-Prosad conjecture for unitary groups using relative trace formula. In establishing the relative trace formula, they needed a form of the Fundamental Lemma comparing the orbital integrals of the standard test functions on the symmetric space $\GL_n(E)/\GL_n(F)$ and on the unitary group $\Ug_n(F)$, here $E/F$ is an unramified extension of the local field $F$ with odd residue characteristic. They explicitly stated (up to sign) a Lie algebra version of this Fundamental Lemma as a conjecture and verified it for $n\leq3$ by explicit computation. Following this idea, Wei Zhang \cite{Z} stated the Lie group version of this Fundamental Lemma as a conjecture and verified it for $n\leq3$.

Let $\sigma$ be the Galois involution of $E$ fixing $F$. Let $\eta_{E/F}:F^\times\to\{\pm1\}$ be the quadratic character associated to the extension $E/F$. Let $\Sg_n(F)$ be the subset of $\GL_n(E)$ consisting of $A$ such that $A\sigma(A)=1$; let $\Ug_n(F)$ be the unitary group associated with the Hermitian space $E^n$ with trivial disciminant. We also need the Lie algebra counterparts of the above spaces. Let $\frs_n(F)$ be the set of $n$-by-$n$ matrices with entries in $E^{-}$ (purely imaginary elements in the quadratic extension $E/F$); let $\fru_n(F)$ be the set of $n$-by-$n$ skew-Hermitian matrices with entries in $E$ (the Hermitian form on $E^n$ has trivial discriminant). For a subset of $K\subset\frgl_n(E)$, let $\one_K$ deonte the characteristic function of $K$. The two versions of the Jacquet-Rallis conjecture are the following identities of orbital integrals:

\begin{conj}\label{cj:OI}
\begin{enumerate}
\item []
\item (Jacquet-Rallis, \cite{JR}) For {\em strongly regular semisimple elements} (see Definition \ref{def:rs}) $A\in\frs_n(F)$ and $A'\in\fru_n(F)$ that {\em match each other} (see Definition \ref{def:match}), we have:
\begin{equation}\label{eq:OI}
\OI_{A}^{\GL_{n-1},\eta}(\one_{\frs_n(\calO_F)})=(-1)^{v(A)}\OI_{A'}^{\Ug_{n-1}}(\one_{\fru_n(\calO_F)}).
\end{equation}

\item (Wei Zhang) For strongly regular semisimple elements $A\in\Sg_n(F)$ and $A'\in\Ug_n(F)$ that match each other (same notion as above), we have:
\begin{equation}\label{eq:gpOI}
\OI_{A}^{\GL_{n-1},\eta}(\one_{\Sg_n(\calO_F)})=(-1)^{v(A)}\OI_{A'}^{\Ug_{n-1}}(\one_{\Ug_n(\calO_F)}).
\end{equation}
\end{enumerate}
Here, for elements $A,A'\in\frgl_n(E)$, and smooth compactly supported functions $f,f'$ on $\frgl_n(E)$,
\begin{eqnarray*}
\OI_{A}^{\GL_{n-1},\eta}(f)&=&\int_{\GL_{n-1}(F)}f(g^{-1}Ag)\eta_{E/F}(\det(g))dg\\
\OI_{A'}^{\Ug_{n-1}}(f')&=&\int_{\Ug_{n-1}(F)}f'(g^{-1}A'g)dg
\end{eqnarray*}
For the definition of $v(A)$, see Definition \ref{def:v}.
\item In the above two situations, if $A$ does not match any $A'$, then the LHS of \eqref{eq:OI} or \eqref{eq:gpOI} is zero.
\end{conj}

Let $\Hg_n(F)$ be the set of Hermitian matrices in $\GL_{n}(E)$ and $\frh_n(F)$ be the set of $n$-by-$n$ Hermitian matrices in $\frgl_n(E)$ with respect to the chosen Hermitian form with triival disciminant. We have the following variant of the Conjecture \ref{cj:OI}:
\begin{conj}\label{cj:OI'}
\begin{enumerate}
\item []
\item For strongly regular semisimple elements $A\in\frgl_n(F)$ and $A'\in\frh_n(F)$ that {\em match each other} (in the similar sense as Definition \ref{def:match}), we have
\begin{equation}\label{eq:algOI'}
\OI_{A}^{\GL_{n-1},\eta}(\one_{\frgl_n(\calO_F)})=(-1)^{v(A)}\OI_{A'}^{\Ug_{n-1}}(\one_{\frh_n(\calO_F)}).
\end{equation}
\item For strongly regular semisimple elements $A\in\GL_n(F)$ and $A'\in\Hg_n(F)$ that match each other, we have
\begin{equation}\label{eq:gpOI'}
\OI_{A}^{\GL_{n-1},\eta}(\one_{\GL_n(\calO_F)})=(-1)^{v(A)}\OI_{A'}^{\Ug_{n-1}}(\one_{\Hg_n(\calO_F)}).
\end{equation}
\item In the above two situations, if $A$ does not match any $A'$, then the LHS of \eqref{eq:algOI'} or \eqref{eq:gpOI'} is zero.
\end{enumerate}
\end{conj}

\subsection{Main results} The main purpose of the paper is to prove the above Conjectures in the case $F$ is a local function field (i.e., of the form $k((\varpi))$ for some finite field $k$) and $\char(F)>n$ (see Corollary \ref{c:main}). 

We first do some reductions. In fact, as observed by Xinyi Yuan, \eqref{eq:algOI'} is simply equivalent to Conjecture \eqref{eq:OI} because multiplication by a purely imaginary element in $E^-\backslash0$ interchanges the situation. In Proposition \ref{p:gptoalg}, we show that the group version \eqref{eq:gpOI} follows from the Lie algebra version \eqref{eq:OI} for any $F$ (of any characteristic); the same argument shows that \eqref{eq:gpOI'} follows from \eqref{eq:algOI'}. Therefore, the orbital integral identity \eqref{eq:OI} for Lie algebras implies all the other identities, for any $F$. Moreover, the vanishing result in Conjecture \ref{cj:OI}(3) and \ref{cj:OI'}(3) follows from a cancellation argument (see Lemma \ref{l:cancel}).

To prove \eqref{eq:OI} in the case $\char(F)>n$, we follow the strategy of the proof of the Langlands-Shelstad Fundamental Lemma in the Lie algebra and function field case, which is recently finished by Ng\^o Bao Ch\^au (\cite{NFL}), building on the work of many mathematicians over the past thirty years. The geometry involved in the Langlands-Shelstad Fundamental Lemma consists of a local part---the affine Springer fibers (cf. \cite{GKM}) and a global part---the Hitchin fibration (cf. Ng\^o \cite{NgoHit}, \cite{NFL}). Roughly speaking, the motives of the affine Springer fibers, after taking Frobenius traces, give the orbital integrals; the motives of the Hitchin fibers can be written as a product of the motives of affine Springer fibers. The advantage of passing from local to global is that we can control the ``bad'' (non-computable) orbital integrals by ``nice'' (computable) orbital integrals, using global topological machinery such as perverse sheaves.

In the following two subsections, we reformulate Conjecture \ref{cj:OI}(1) using local and global moduli spaces, and indicate the main ideas of the proof.

\subsection{The local reformulation}
As the first step towards a local reformulation, we translate the problem of computing orbital integrals into that of counting lattices. In \cite{JR}, the authors introduced $(2n-1)$-invariants associated to an element $A\in\frgl_n(E)$ with respect to the conjugation action of $\GL_{n-1}(E)$. Here the embedding $\GL_{n-1}(E)\hookrightarrow\GL_n(E)$ is given by a splitting of the vector space $E^n=E^{n-1}\oplus E$. The splitting gives a distinguished vector $e_0$ (which spans the one dimension direct summand) and a distinguished covector $e_0^*$ (the projection to $Ee_0$). In this paper, we use a different (but equivalent) set of invariants $a=(a_1,\cdots,a_n)$ and $b=(b_0,\cdots,b_{n-1})$ for $A$: the $a_i$'s are the coefficients of the characteristic polynomial of $A$ and $b_i=e_0^*A^ie_0$ (so that $b_0=1$). We say that $A\in\frs_n(F)$ and $A'\in\fru_n(F)$ {\em match each other} if they have the same collection of invariants viewed as elements in $\frgl_n(E)$.

We fix a collection of invariants $(a,b)$ which is integral and strongly regular semisimple (see Definition \ref{def:rs} and Lemma \ref{l:eqrs}). Then we can associate a finite flat $\calO_F$-algebra $R_a$ (see \S \ref{ss:inv}). The invariants $b$ gives an $R_a$-linear embedding $\gamma_{a,b}:R_a\hookrightarrow R_a^\vee=\Hom_{\calO_F}(R_a,\calO_F)$. We can rewrite the LHS of (\ref{eq:OI}) as
\begin{equation*}
\OI_{A}^{\GL_{n-1},\eta}(\one_{\frs_n(\calO_F)})=\pm\sum_{i}(-1)^i\#\sMloc_{i,a,b}
\end{equation*}
where each $\sMloc_{i,a,b}$ is the set of $R_a$-lattices $\Lambda$ such that $R_a\subset\Lambda\subset R_a^\vee$ and $\leng_{\calO_F}(R_a^\vee:\Lambda)=i$ (see Notation \ref{n:ring}).

The $E$-vector space $R_a(E)=R_a\otimes_{\calO_F}E$ carries a natural Hermitian form given also by $b$. Similarly, we can rewrite the RHS of (\ref{eq:OI}) as
\begin{equation*}
\OI_{A'}^{\Ug_{n-1}}(\one_{\fru_n(\calO_F)})=\#\sNloc_{a,b}
\end{equation*}
where $\sNloc_{a,b}$ is the set of $R_a(\calO_E)$-lattices $\Lambda'$, self-dual under the Hermitian form, and $R_a(\calO_E)\subset\Lambda'\subset R_a^\vee(\calO_E)$.

When $F$ is a function field with residue field $k$, there are obvious moduli spaces of lattices $\Mloc_{i,a,b}$ and $\Nloc_{a,b}$ defined over $k$ such that $\sMloc_{i,a,b}$ and $\sNloc_{a,b}$ are the set of $k$-points of $\Mloc_{i,a,b}$ and $\Nloc_{a,b}$. By Lefschetz trace formula for schemes over $k$, the Conjecture \ref{cj:OI} is a consequence of the following theorem, which is the main local result of the paper

\begin{mthloc}
Suppose $\char(F)=\char(k)>\max\{n,2\}$ and $\eta_{E/F}(\Delta_{a,b})=1$. Then there is an isomorphism of graded $\Frob_k$-modules:
\begin{equation*}
\bigoplus_{i=0}^{\val_F(\Delta_{a,b})}H^*(\geom{\Mloc_{i,a,b}}{k},\Ql(\eta_{k'/k})^{\otimes i})\cong H^*(\geom{\Nloc_{a,b}}{k},\Ql)
\end{equation*}
\end{mthloc}

For more details on the notation, see \S \ref{ss:locmatch}. This theorem is deduced from the global Main Theorem, which we shall give an overview in the next subsection.

We point out that there are several easy cases of the Conjecture \ref{cj:OI} which are verified in \S \ref{ss:FL} without any restriction on $F$.

\subsection{The global approach}
The global geometry related to the Jacquet-Rallis Fundamental Lemma is the modified versions of Hitchin moduli spaces for the groups $\GL_n$ and $\Ug_n$. The geometry of Hitchin fibrations is studied in great detail by Ng\^o \cite{NFL} and Laumon-Ng\^o \cite{LN}, the latter treats the unitary group case and is especially important for the purpose of this paper.

We fix a smooth projective and geometrically connected curve $X$ over $k$; to study the unitary group, we also fix an \'etale double cover $X'\to X$. We fix two effect divisors $D$ and $D_0$ on $X$ of large enough degrees.

We introduce the stack $\calM$, classifying quadruples $(\calE,\phi,\lambda,\mu)$ where $\calE$ is a vector bundle of rank $n$; $\phi$ is a Higgs field on $\calE$; $\lambda:\calO_X(-D_0)\to\calE$  is the global counterpart of the distinguished vector $e_0$ and $\mu:\calE\to\calO_X(D_0)$ the global counterpart of the distinguished covector $e_0^*$. The stack $\calM$ is the disjoint union of $\calM_i$ ($i\in\ZZ$), according to the degree of $\calE$.

We also introduce the stack $\calN$, classifying quadruples $(\calE',h,\phi',\mu')$ where $\calE'$ is a vector bundle of rank $n$ on $X'$; $h$ is a Hermitian structure on $\calE'$; $\phi'$ is a skew-Hermitian Higgs field on $\calE'$ and $\mu':\calE\to\calO_{X'}(D_0)$ is the distinguished covector (the distinguished vector is determined by $\mu'$ using the Hermitian structure).

Both $\calM_i$ and $\calN$ fiber over the ``Hitchin base'' $\calA\times\calB$, classifying global invariants $(a,b)$. In \S \ref{s:mod}, we prove the following geometric properties of these moduli spaces
\begin{itemize}
\item (Proposition \ref{p:smGL} and \ref{p:smU}) Over the locus $\Ah\times\Bh$, $\Mh_i$ (for ``non-extremal'' $i$) and $\Nh$ are smooth schemes over $k$ and the ``Hitchin maps'' $\fh_i:\Mh_i\to\Ah\times\Bh$ and $\gh:\Nh\to\Ah\times\Bh$ are proper;
\item (Proposition \ref{p:small}) Fix a Serre invariant $\delta\geq1$ (see \S\ref{ss:small}), for $\deg(D)$ and $\deg(D_0)$ large enough, the restrictions of $\fh_i$ and $\gh$ to $\Adel\times\Bh$ are {\em small} (cf. Notation \ref{n:small}); 
\item (Proposition \ref{p:pdGL} and \ref{p:pdU}) The fibers $\calM_{i,a,b}$ and $\calN_{a,b}$ can be written as products of local moduli spaces $\calM^x_{i_x,a_x,b_x}$ and $\calN^x_{a_x,b_x}$ defined in a similar way as $\Mloc_{i,a,b}$ and $\Nloc_{a,b}$.
\end{itemize}

The global part of the Main Theorem is
\begin{mthglob}
Fix $\delta\geq1$,$\deg(D)\geq c_\delta$ and $n(\deg(D_0)-g+1)\geq\delta+g_Y$. Then there is a natural isomorphism in $D^b_c(\Adel\times\Bh)$:
\begin{equation}\label{eq:gl}
\bigoplus_{i=-d}^d\fh_{i,*}L_{d-i}|_{\Adel\times\Bh}\cong\gh_*\Ql|_{\Adel\times\Bh}.
\end{equation} 
\end{mthglob}
Here $d=n(n-1)\deg(D)/2+n\deg(D_0)$; $L_{d-i}$ is a local system of rank one and order two on $\Mh_{i}$ (see \S \ref{ss:lsonM}), which is a geometric analogue of the factor $\eta_{E/F}(\det(g))$ appearing in the orbital integral $\OI_{A}^{\GL_{n-1},\eta}(\one_{\frs_n(\calO_F)})$.

The smallness of $\fh_i$ and $\gh$ over $\Adel\times\Bh$ and the smoothness of $\Mh_i$ and $\Nh$ are the essential geometric properties that enable one to prove the above theorem. In fact, these two properties imply that both sides are middle extensions of some local system on some dense open subset of $\Adel\times\Bh$, and to verify (\ref{eq:gl}), we only need to restrict to a nice open dense subset where both sides are explicitly computable.

However, there is one technical difficulty. The schemes $\Mh_i$ are visibly smooth only for ``non-extremal'' $i$'s (see Proposition \ref{p:smGL} for the range of $i$ where $\Mh_i$ is proved to be smooth), therefore the above argument only works for these $i$'s. To get around this difficulty, we enlarge $D_0$ and compare the moduli spaces defined by two different $D_0$'s. Indices $i$ which are extremal for the original $D_0$ become non-extremal for the new $D_0$. Details are given in \S \ref{ss:match}.

Finally, we use the global part of the Main Theorem to prove the local part. We identify $F$ with the local field associated to a $k$-point $x_0$ on the curve $X$. For local invariants $(a^0,b^0)$, we may choose global invariants $(a,b)$ such that the global moduli spaces $\calM_{i,a,b}$ and $\calN_{a,b}$, when expressed as product of local moduli spaces, are very simple away from $x_0$. Taking the Frobenius traces of the two sides of (\ref{eq:gl}) and using the product formulae, we get a formula of the following form
\begin{equation*}
\Tr(\Frob_k,M_{x_0})\cdot\prod_{x\neq x_0}\Tr(\Frob_k,M_{x})=\Tr(\Frob_k,N_{x_0})\cdot\prod_{x\neq x_0}\Tr(\Frob_k,N_x)
\end{equation*}
where $M_x$ and $N_x$ are the cohomology groups of the relevant local moduli spaces. It is easy to show $\Tr(\Frob_k,M_x)=\Tr(\Frob_k,N_x)$ for all $x\neq x_0$, but to conclude that $\Tr(\Frob_k,M_{x_0})=\Tr(\Frob_k,N_{x_0})$, which is what we need, it remains to make sure that the terms $\Tr(\Frob_k,N_x)$ are nonzero for all $x\neq x_0$. This is another technical difficulty of the paper, and is responsible for the length of \S \ref{ss:pre}.

\subsection{Plan of the paper}

In \S \ref{s:loc}, we first fix notations and introduce the invariants $(a,b)$. We then reformulate the problem into counting of lattices in \S \ref{ss:OIGL} and \S \ref{ss:OIU}. We verify a few simple cases of Conjecture \ref{cj:OI} in \S \ref{ss:FL}, including Conjecture \ref{cj:OI}(3). We deduce the group version Conjecture \ref{cj:OI}(2) from the Lie algebra version Conjecture \ref{cj:OI}(1) in \S \ref{ss:gptoalg}. Except for \S \ref{ss:locmatch}, we work with no restrictions on $F$.

In \S \ref{s:mod}, we introduce the global moduli spaces $\calM_i$ and $\calN$ and study their geometric properties.

In \S \ref{s:match}, we formulate and prove the global part of the Main Theorem. To this end, we need to study perverse sheaves on the symmetric powers of curves, especially the ``binomial expansion'' formula (Lemma \ref{l:bin}), which demystifies the decomposition (\ref{eq:gl}).

In \S \ref{s:pf}, we deduce the local part of the Main Theorem from the global part.

\subsection{Notations}

\subsubsection{}\label{n:ring} Let $\calO$ be a commutative ring. For a scheme $X$ over $\calO$ or an $\calO$-module $M$, and an $\calO$-algebra $R$, we let $X(R)$ be the $R$-points of $X$ and $M(R):=M\otimes_{\calO}R$.

For a DVR $\calO$ and two full rank $\calO$-lattice $\Lambda_1,\Lambda_2$ in some $\Frac(\calO)$-vector space $V$, we define the relative length $\leng_{\calO}(\Lambda_1:\Lambda_2)$ to be
\begin{equation*}
\leng_{\calO}(\Lambda_1:\Lambda_2):=\leng_{\calO}(\Lambda_1/\Lambda_1\cap\Lambda_2)-\leng_{\calO}(\Lambda_2/\Lambda_1\cap\Lambda_2)
\end{equation*}
where $\leng_{\calO}(-)$ denotes the usual length of a torsion $\calO$-module.

\subsubsection{}\label{n:sh} Coherent sheaves are denoted by the calligraphic letters $\calE,\calF,\calL,\cdots$; constructible $\Ql$-complexes are denoted by the capital letters $L,K,\cdots$.

\subsubsection{}\label{n:semilocal} From \S 3, we will work with a fixed smooth base curve $X$ over a field $k$. For a morphism $p:Y\to X$ and a closed point $x\in X$, we denote by $\calO_{Y,x}$ the completed semilocal ring of $\calO_Y$ along $p^{-1}(x)$.

If $Y$ is a Gorenstein curve, let $\omega_{Y/X}$ be the relative dualizing sheaf
\begin{equation*}
\omega_{Y/X}=\omega_{Y/k}\otimes p^*\omega^{-1}_{X/k}.
\end{equation*}
For a coherent sheaf $\calF$ on $Y$, let $\calF^\vee=\underline{\Hom}_{Y}(\calF,\omega_{Y/X})$ be the (underived) Grothendieck-Serre dual. When we work over an extra parameter scheme $S$ so that $p:Y\to X\times S$, then $^\vee$ means $\underline{\Hom}_{Y}(-,\omega_{Y/X\times S})$.

\subsubsection{}\label{n:loc} For an \'etale double cover (a finite \'etale map of degree 2) $\pi:X'\to X$ of a scheme $X$, we decompose the sheaf $\pi_*\Ql$ into $\pm1$-eigenspaces under the natural action of $\ZZ/2\cong\Aut(X'/X)$:
\begin{equation*}
\pi_*\Ql=\Ql\oplus L.
\end{equation*}
The the rank one local system $L$ satisfies $L^{\otimes2}\cong\Ql$. We call $L$ the local system {\em associated to} the \'etale double cover $\pi$.
 
\subsubsection{}\label{n:mid} We use the terminology ``middle extension'' in a non-strict way. If $K$ is a $\Ql$-complex on a scheme $X$, we say $K$ is a {\em middle extension on $X$} if for some (and hence any) smooth open dense subset $j:U\hookrightarrow X$ over which $K$ is a local system placed at degree $n$, we have
\begin{equation*}
K\cong j_{!*}(j^*K[n+\dim X])[-n-\dim X].
\end{equation*}

\subsubsection{}\label{n:small} Recall from \cite[\S 6.2]{GM} that a proper surjective morphism $f:Y\to X$ between {\em irreducible} schemes over an algebraically closed field $\Omega$ is called {\em small} if for any $r\geq1$, we have
\begin{equation}\label{eq:cosm}
\codim_{X}\{x\in X|\dim f^{-1}(x)\geq r\}\geq 2r+1.
\end{equation}
We will use this terminology in a loose way: we will not require $Y$ to be irreducible (but we do require all other conditions). This is just for notational convenience. The main property that we will use about small morphisms is:

{\em If $Y/\Omega$ is smooth and equidimensional, then $f_*\Ql$ is a middle extension on $X$.} 

\subsubsection{}\label{n:Fr} For a finite field $k$, we denote by $\Frob_k$ the {\em geometric} Frobenius element in $\Gal(\overline{k}/k)$.

\subsection*{Acknowledgment} The author would like to thank Xinyi Yuan and Wei Zhang for suggesting the problem and for helpful discussions. He also thanks Joseph Bernstein, Julia Gordon, Yifeng Liu and Shou-Wu Zhang for useful suggestions. Different parts of this work were conceived during the workshop on Trace Formula held in Columbia University in Dec. 2008, and the AIM workshop on Relative Trace formula and Periods of Automorphic Forms in Aug. 2009. The author would like to thank the organizers of these workshops.

\section{Local formulation}\label{s:loc}

\subsection{The setting}\label{ss:setting}

Let $F$ be a non-Archimedean local field with valuation ring $\calO_F$, uniformizing parameter $\varpi$ and residue field $k=\FF_q$ such that $q$ is odd. Let $E$ be either the unramified quadratic extension of $F$ (in which case we call $E/F$ is nonsplit) or $E=F\times F$ (in which case we call $E/F$ is split). Let $\calO_E$ be the valuation ring of $E$ and $k'$ be the residue algebra of $k$. Let $\sigma$ be the generator of $\Gal(E/F)$. Let $\eta_{E/F}:F^\times/\Nm E^\times\to\{\pm1\}$ be the quadratic character associated to the extension $E/F$: this is trivial if and only if $E/F$ is split. We decompose $E$ and $\calO_E$ according into eigenspaces of $\sigma$:
\begin{equation*}
E=F\oplus E^-;\hspace{1cm} \calO_E=\calO_F\oplus\calO_E^-
\end{equation*}
where $\sigma$ acts on $E^-$ and $\calO_E^-$ by $-1$.

We fix a free $\calO_F$-module $W$ of rank $(n-1)$. Let $V=W\oplus \calO_F\cdot e_0$ be a free $\calO_F$-module of rank $n$ with a distinguished element $e_0$ and let $e_0^*:V\to \calO_F$ be the projection along $W$ such that $e_0^*(e_0)=1$. Let $W^\vee=\Hom_{\calO_F}(W,\calO_F)$ and $V^\vee=\Hom_{\calO_F}(V,\calO_F)$.

Let $\GL_{n-1}=\GL(W)$ and $\GL_{n}=\GL(V)$ be the general linear groups over $\calO_F$. We have the natural embedding $\GL_{n-1}\hookrightarrow\GL_{n}$ as block-diagonal matrices:
\begin{equation*}
A\mapsto\left(\begin{array}{lll} A & \\ & 1\end{array}\right)
\end{equation*}

Let $\frgl_n$ be the Lie algebra (over $\calO_F$) of $\GL_n$ consisting of $\calO_F$-linear operators on $V$. Let 
\begin{equation*}
\frs_n(\calO_F):=\{A\in\frgl_n(\calO_E)|A+\sigma(A)=0\}.
\end{equation*}
Then $\frs_n(F)$ is be the subset of $\frgl_n(E)$ consisting of matrices with entries in $E^-$. The group $\GL_n$ acts on $\frs_n$ by conjugation.

Let $(\cdot,\cdot)$ be a Hermitian form on $W$ with trivial discriminant and extend this to a Hermitian form on $V$ by requiring that $(W,e_0)=0$ and $(e_0,e_0)=1$. These Hermitian forms define unitary groups $\Ug_{n-1}=\Ug(W,(\cdot,\cdot))$ and $\Ug_{n}=\Ug(V,(\cdot,\cdot))$ over $\calO_F$. We also have the natural embedding $\Ug_{n-1}\hookrightarrow\Ug_{n}$ as block-diagonal matrices:
\begin{equation*}
A\mapsto\left(\begin{array}{lll} A & \\ & 1\end{array}\right)
\end{equation*}

Let $\fru_{n}$ be the Lie algebra (over $\calO_F$) of $\Ug_{n}$, i.e.,
\begin{equation*} 
\fru_n(\calO_F)=\{A\in\frgl_n(\calO_E)|A+A^{\#}=0\}
\end{equation*}
where $A^{\#}$ is the adjoint of $A$ under the Hermitian form $(\cdot,\cdot)$. Then $U_n$ acts on $\fru_n$ by conjugation.

\subsection{The invariants}\label{ss:inv}
For any $A\in\frgl_n(E)$, and $i=1,\cdots,n$, let
\begin{equation*}
a_i(A):=\Tr(\bigwedge^iA)
\end{equation*}
be the coefficients of the characteristic polynomial of $A$, i.e.,
\begin{equation*}
\det(t\id_V-A)=t^n+\sum_{i=1}^n(-1)^ia_i(A)t^{n-i}.
\end{equation*}
For $i=0,\cdots,n-1$, let
\begin{equation*}
b_i(A)=e_0^*(A^ie_0).
\end{equation*}
so that $b_0=1$. The $2n$-tuple $(a(A),b(A))$ of elements in $E$ are called the {\em invariants} of $A\in\frgl_n(E)$. It is easy to check that these are invariants of $\frgl_n(E)$ under the conjugation action by $\GL_{n-1}(E)$.

\begin{defn}\label{def:rs}
An element $A\in\frgl_n(E)$ is said to be {\em strongly regular semisimple with respect to the $\GL_{n-1}(E)$-action}, if
\begin{enumerate}
\item\label{it:rss} $A$ is regular semisimple as an element of $\frgl_n(E)$;
\item\label{it:cyc} the vectors $\{e_0,Ae_0,\cdots,A^{n-1}e_0\}$ form an $E$-basis of $V(E)$;
\item\label{it:dualcyc} the vectors $\{e_0^*,e_0A,\cdots,e_0^*A^{n-1}\}$ form an $E$-basis of $V^\vee(E)$.
\end{enumerate}
\end{defn}

\begin{defn}\label{def:v}
For $A\in\frgl_n(E)$, define $v(A)\in\ZZ$ to be the $F$-valuation of the $n$-by-$n$ matrix formed by the row vectors $\{e_0^*,e_0^*A,\cdots,e_0^*A^{n-1}\}$ under an $\calO_F$-basis of $V^\vee$.
\end{defn}

For a collection of invariants $(a,b)\in E^{2n}$ (we allow general $b_0\in E$, not just 1), we introduce a finite $E$-algebra
\begin{equation*}
R_{a}(E):=E[t]/(t^n-a_1t^{n-1}+\cdots+(-1)^na_n)
\end{equation*}
Let $R_a^\vee(E):=\Hom_E(R_{a}(E),E)$ be its linear dual, which is naturally an $R_{a}(E)$-module. The data of $b$ gives the following element $b'\in R_a^\vee(E)$:
\begin{eqnarray}\label{eq:bprime}
b':R_a(E)&\to&E\\
\notag
t^i&\mapsto&b_i,\hspace{1cm}i=0,1,\cdots,n-1.
\end{eqnarray}
which induces an $R_{a}(E)$-linear homomorphism
\begin{equation*}
\gamma'_{a,b}:R_{a}(E)\to R_{a}^\vee(E).
\end{equation*}
In other words, $\gamma'_{a,b}$ is given by the pairing:
\begin{eqnarray}\label{eq:pairing}
R_{a}(E)\otimes_{E}R_{a}(E)&\to&E\\
\notag
(x,y)\mapsto b'(xy).
\end{eqnarray}

\begin{defn}\label{def:Delta}
The $\Delta$-invariant $\Delta_{a,b}$ of the collection $(a,b)\in E^{2n}$ is the determinant of the map $\gamma'_{a,b}$ under the $E$-basis $\{1,t,\cdots,t^{n-1}\}$ of $R_{a}(E)$ and the corresponding dual basis of $R_{a}^\vee(E)$.
\end{defn}

It is easy to see that for $A\in\frgl_n(E)$ with invariants $(a,b)$, $\Delta_{a,b}$ is the determinant of the matrix $(e_0^*A^{i+j}e_0)_{0\leq i,j\leq n}$.

\begin{lemma}\label{l:eqrs}
Let $A\in\frgl_n(E)$ with invariants $(a,b)\in E^{2n}$. Then $A$ is strongly regular semisimple if and only if
\begin{enumerate}
\item $R_{a}(E)$ is an \'etale algebra over $E$;
\item $\Delta_{a,b}\neq0$; i.e., $\gamma'_{a,b}:R_{a}(E)\to R_{a}^\vee(E)$ is an isomorphism.
\end{enumerate}
\end{lemma}
\begin{proof}
Condition (1) is equivalent to that $A$ is regular semisimple as an element of $\frgl_n(E)$. Condition (2) is equivalent to that the matrix $(e_0^*A^{i+j}e_0)_{0\leq i,j\leq n}$ is nondegenerate. Since this matrix is the product of the two matrices $(e_0^*,e_0^*A,\cdots,e_0^*A^{n-1})$ and $(e_0,Ae_0,\cdots,A^{n-1}e_0)$, the nondegeneracy of $(e_0^*A^{i+j}e_0)_{0\leq i,j\leq n}$ is equivalent to the nondegeneracy of $(e_0^*,e_0^*A,\cdots,e_0^*A^{n-1})$ and $(e_0,Ae_0,\cdots,A^{n-1}e_0)$, which are the last two conditions of Definition \ref{def:rs}.
\end{proof}

\begin{remark}
From the above Lemma, we see that the strong regular semisimplicity of $A\in\frgl_n(E)$ is in fact a property of its invariants $(a,b)$. Therefore, we call a collection of invariants $(a,b)$ {\em strongly regular semisimple}, if it satisfies the conditions in Lemma \ref{l:eqrs}.
\end{remark}

If the invariants $(a,b)$ are elements in $\calO_E$, we get a canonical $\calO_E$-form $R_{a}(\calO_E)$ of $R_a(E)$ by setting 
\begin{equation}\label{eq:RaE}
R_{a}(\calO_E):=\calO_E[t]/(t^n-a_1t^{n-1}+\cdots+(-1)^na_n)
\end{equation}
Let $R_{a}^\vee(\calO_E)=\Hom_{\calO_E}(R_{a}(\calO_E),\calO_E)$ be its dual. Then $\gamma'_{a,b}$ restricts to an $R_{a}(\calO_E)$-linear map
\begin{equation*}
\gamma'_{a,b}:R_{a}(\calO_E)\to R_{a}^\vee(\calO_E).
\end{equation*}

For $A\in\frs_n(F)$ or $\fru_n(F)$ with invariants $(a,b)$, it is obvious that $a_i,b_i\in E^{\sigma(-1)^i}$. Suppose further more that $a_i,b_i\in\calO_E^{\sigma=(-1)^i}$. We extend the involution $\sigma$ on $\calO_E$ to an involution $\sigma_R$ on $R_{a}(\calO_E)$ by requiring $\sigma_R(t)=-t$. The involution $\sigma_R$ defines an $\calO_F$-form of $R_{a}(\calO_E)$:
\begin{equation*}
R_{a}:=R_{a}(\calO_E)^{\sigma_R}
\end{equation*}
Let $R_{a}^\vee:=\Hom_{\calO_F}(R_{a},\calO_F)$. The map $\gamma'_{a,b}$ restricts to an $R_{a}$-linear homomorphism
\begin{equation}\label{eq:gamma}
\gamma_{a,b}:R_{a}\to R_{a}^\vee.
\end{equation}
and $\gamma_{a,b}'=\gamma_{a,b}\otimes_{\calO_F}\calO_E$. Since $R_a$ is an $\calO_F$-form of $R_a(\calO_E)$ defined before, there is no confusion in using the notations $R_a(\calO_E)$ or $R_a(E)$ (cf. Notation \ref{n:ring}).

\subsection{Orbital integrals for $\frs_n(F)$}\label{ss:OIGL}
Let $A\in\frs_n(F)$ be strongly regular semisimple with invariants $(a,b)$. Let
\begin{equation*}
\OI_{A}^{\GL_{n-1},\eta}(\one_{\frs_n(\calO_F)}):=\int_{\GL_{n-1}(F)}\one_{\frs_n(\calO_F)}(g^{-1}Ag)\eta_{E/F}(\det(g))dg
\end{equation*}
where $dg$ is the Haar measure on $\frs_n(F)$ such that $\vol(\frs_n(\calO_F),dg)=1$ and $\one_{\frs_n(\calO_F)}$ is the characteristic function of $\frs_n(\calO_F)\subset\frs_n(F)$.

\begin{remark}
It is easy to see if the orbital integral $\OI_{A}^{\GL_{n-1},\eta}(\one_{\frs_n(\calO_F)})\neq0$, then $a_i,b_i\in\calO_E^{\sigma=(-1)^i}$.
\end{remark}

Now we suppose $(a,b)$ is strongly regular semisimple and $a_i,b_i\in\calO_E^{\sigma=(-1)^i}$. We view $R_{a}$ as a sublattice of $R_{a}^\vee$ via the map $\gamma_{a,b}:R_{a}\hookrightarrow R_{a}^\vee$. Then
\begin{equation*}
\leng_{\calO_F}(R_a^\vee:R_a)=\val_F(\Delta_{a,b}).
\end{equation*}

For each integer $0\leq i\leq\val_F(\Delta_{a,b})$, let
\begin{equation*}
\sMloc_{i,a,b}:=\{R_a\textup{-lattices }\Lambda|R_a\subset\Lambda\subset R_a^\vee\textup{ and }\leng_{\calO_F}(R_a^\vee:\Lambda)=i\}
\end{equation*}

\begin{prop}\label{p:GLloc} Let $A\in\frs_n(F)$ be strongly regular semisimple with invariants $a_i,b_i\in\calO_E^{\sigma=(-1)^i}$, then
\begin{equation*}
\OI^{\GL_{n-1},\eta}_A(\one_{\frs_n(\calO_F)})=\eta_{E/F}(\varpi)^{v(A)}\sum_{i=0}^{\val_F(\Delta_{a,b})}\eta_{E/F}(\varpi)^i\#\sMloc_{i,a,b}.
\end{equation*}
\end{prop}
\begin{proof}
By Definition \ref{def:rs}(\ref{it:cyc}), we have a $\sigma$-equivariant $E$-linear isomorphism:
\begin{eqnarray}\label{eq:iota}
\iota':R_{a}(E)&\isom& V(E)\\
t^i&\mapsto& A^ie_0.
\end{eqnarray}
Therefore $\iota'$ restrict to an $F$-linear isomorphism 
\begin{equation*}
\iota:R_a(F)\isom V(F).
\end{equation*}
Define $R_{a,W}:=R_a\cap\iota^{-1}(W(F))$ and $R_{a,W}^\vee:=R_a^\vee\cap\iota^{-1}(W(F))$. We define some auxiliary sets
\begin{eqnarray*}
X_{i,A}&:=&\{g\in\GL_{n-1}(F)/\GL_{n-1}(\calO_F)|g^{-1}Ag\in\frs_n(\calO_F),\val_F(\det(g))=i\};\\
X'_{i,A}&:=&\{\calO_F\textup{-lattices }L\subset W(F)|A(L)\subset\calO_E^-L,\leng_{\calO_F}(L:W)=i\};\\
M^W_{i,a,b}&:=&\{\calO_F\textup{-lattices }\Lambda_W\subset R_{a,W}(F)|\Lambda_W\oplus\calO_F1_R\subset R_a(F)\textup{ is stable under }R_a,\\
&&\hspace{1cm}\leng_{\calO_F}(R_{a,W}^\vee:\Lambda_W)=i\}.
\end{eqnarray*}
where $1_R\in R_a$ is the identity element.

Note that the group $\GL_{n-1}(F)$ acts transitively on the set of $\calO_F$-lattices in $W(F)$ by left translation, and the stabilizer of $W$ equal to $\GL_{n-1}(\calO_F)$. Therefore we get a bijection
\begin{eqnarray*}
X_{i,A}&\isom&X'_{i,A}\\
g&\mapsto& gW.
\end{eqnarray*}

We identify $R_a\stackrel{\gamma}{\hookrightarrow}R_a^\vee$ both as $\calO_F$-lattices in $V(E)$ via $\iota$. Observe that
\begin{equation*}
\leng_{\calO_F}(R_{a,W}^\vee:W)=\leng_{\calO_F}(R_a^\vee:V)=v(A),
\end{equation*}
therefore we have a bijection
\begin{eqnarray*}
M^W_{v(A)-i,a,b}&\isom&X'_{i,A}\\
\Lambda_W&\mapsto&\iota(\Lambda_W).
\end{eqnarray*}
Finally we have a bijection
\begin{eqnarray*}
M^W_{i,a,b}&\isom&\sMloc_{i,a,b}\\
\Lambda_W&\mapsto&\Lambda_W\oplus\calO_F1_R.
\end{eqnarray*}
We check this is a bijection. On one hand, for $\Lambda_W\in M^W_{i,a,b}$, we have
\begin{equation*}
R_a\subset R_a(\Lambda_W\oplus\calO_F1_R)\subset\Lambda_W\oplus\calO_F1_R.
\end{equation*}
We also have
\begin{equation*}
b'(R_a(\Lambda_W\oplus\calO_F1_R))\subset b'(\Lambda_W\oplus\calO_F1_R)\subset\calO_F,
\end{equation*}
therefore $\Lambda_W\oplus\calO_F1_R\subset R_a^\vee$. This verifies $\Lambda_W\oplus\calO_F1_R\in\sMloc_{i,a,b}$.

On the other hand, we have to make sure that every $\Lambda\in\sMloc_{i,a,b}$ has the form $\Lambda=\Lambda_W\oplus\calO_F1_R$ for some lattice $\Lambda_W\subset R_{a,W}(F)$, i.e., $\calO_F1_R\subset\Lambda$ is saturated. But we can factorize the identity map on $\calO_F$ as
\begin{equation*}
\calO_F\xrightarrow{1\mapsto1_R}R_a\xrightarrow{\gamma}R_a^\vee\xrightarrow{\ev(1_R)}\calO_F.
\end{equation*}
Therefore for $R_a\subset\Lambda\subset R_a^\vee$, $\calO_F1_R\subset\Lambda$ is always saturated.

Now that we have set up a bijection between $X_{i,A}$ and $\sMloc_{v(A)-i,a,b}$, we have
\begin{eqnarray*}
\OI^{\GL_{n-1},\eta}_A(\one_{\frs_n(\calO_F)})&=&\sum_{i}\eta_{E/F}(\varpi)^i\#X_{i,A}\\
&=&\eta_{E/F}(\varpi)^{v(A)}\sum_{i}\eta_{E/F}(\varpi)^i\#\sMloc_{i,a,b}.
\end{eqnarray*}
\end{proof}

\subsection{Orbital integrals for $\fru_n(F)$}\label{ss:OIU}

\begin{remark}
It is easy to see if the orbital integral $\OI_{A}^{\Ug_{n-1}}(\one_{\frs_n}(\calO_F))\neq0$, then $a_i,b_i\in\calO_E^{\sigma=(-1)^i}$.
\end{remark}

Now we suppose $(a,b)$ is strongly regular semisimple and $a_i,b_i\in\calO_E^{\sigma=(-1)^i}$. We identify $R_{a}(\calO_E)$ as a sublattice of $R_{a}^\vee(\calO_E)$ via $\gamma'_{a,b}$. Recall from (\ref{eq:iota}) that we have an isomorphism $\iota':R_a(E)\isom V(E)$. The transport of the Hermitian form $(\cdot,\cdot)$ to $R_a(E)$ via $\iota'$ is given by:
\begin{equation}\label{eq:HermR}
(x,y)_R=b'(x\sigma_R(y)).
\end{equation}
where $b':R_a(E)\to E$ is defined in (\ref{eq:bprime}).

\begin{remark}\label{r:evenval}
Since the Hermitian form $(\cdot,\cdot)$ has trivial discriminant, so is $(\cdot,\cdot)_R$. Therefore if $A'\in\fru_n(F)$ has invariants $(a,b)$, then $\eta_{E/F}(\Delta_{a,b})=1$.
\end{remark}

Recall that for an $\calO_E$-lattice $\Lambda'\subset R_{a}(E)$, the {\em dual lattice under the Hermitian form $(\cdot,\cdot)_R$} is the $\calO_E$-lattice
\begin{equation*}
\Lambda'^{\bot}:=\{x\in R_{a}(E)|(x,\Lambda')_R\subset\calO_E\}.
\end{equation*}
Such a lattice is called {\em self-dual} (under the given Hermitian form) if $\Lambda'^{\bot}=\Lambda'$. Comparing with the pairing (\ref{eq:pairing}), it is easy to see that $R_{a}^\vee(\calO_E)$ is the dual of $R_{a}(\calO_E)$ under the Hermitian form $(\cdot,\cdot)_R$. We define
\begin{equation*}
\sNloc_{a,b}:=\{\textup{self-dual }R_{a}(\calO_E)\textup{-lattices }\Lambda'|R_{a}(\calO_E)\subset\Lambda'\subset R_a^\vee(\calO_E)\}.
\end{equation*}

\begin{prop}\label{p:Uloc}
Let $A'\in\fru_n(F)$ be strongly regular semisimple with invariants $a_i,b_i\in\calO_E^{\sigma=(-1)^i}$, then
\begin{equation*}
\OI_{A'}^{\Ug_{n-1}}(\one_{\fru_n(\calO_F)})=\#\sNloc_{a,b}.
\end{equation*}
\end{prop}
\begin{proof}
The argument is similar as the proof of Proposition \ref{p:GLloc}. Let
\begin{equation*}
Y_{A'}=\{g\in\Ug_{n-1}(F)/\Ug_{n-1}(\calO_F)|g^{-1}A'g\in\fru_n(\calO_F)\}.
\end{equation*}
Then $\Ug_{n-1}(F)$ acts transitively on the set of self-dual $\calO_E$-lattices in $W(E)$, such that the stabilizer of $W(\calO_E)$ is $\Ug_{n-1}(\calO_F)$. Therefore we get a bijection
\begin{eqnarray*}
Y_{A'}&\isom&\sNloc_{a,b}\\
g&\mapsto&\iota'^{-1}(gW(\calO_E))\oplus\calO_E1_R.
\end{eqnarray*}
Hence
\begin{equation*}
\OI_{A'}^{\Ug_{n-1}}(\one_{\fru_n(\calO_F)})=\#Y_{A'}=\#\sNloc_{a,b}.
\end{equation*}
\end{proof}

\subsection{The Fundamental Lemma and simple cases}\label{ss:FL}
\begin{defn}\label{def:match}
Two strongly regular semisimple elements $A\in\frs_n(F)$ and $A'\in\fru_n(F)$ are said to {\em match each other}, if they have the same invariants.
\end{defn}

Now we have explained all the notions appearing in Conjecture \ref{cj:OI}. By Proposition \ref{p:GLloc} and \ref{p:Uloc}, the Conjecture \ref{cj:OI}(1)(3) are implied by
\begin{conj}\label{cj:count}
For any strongly regular semisimple collection of invariants $(a,b)$ such that $a_i,b_i\in\calO_E^{\sigma=(-1)^i}$ (in particular we allow arbitrary $b_0\in\calO_F$), we have
\begin{equation}\label{eq:count}
\sum_{i=0}^{\val_F(\Delta_{a,b})}\eta_{E/F}(\varpi)^i\#\sMloc_{i,a,b}=\#\sNloc_{a,b}.
\end{equation}
\end{conj}

In the rest of this subsection, we prove some easy cases of the Conjecture \ref{cj:count} by straight-forward counting argument.

\begin{lemma}\label{l:cancel}
The Conjecture \ref{cj:count} is true if $\eta_{E/F}(\Delta_{a,b})\neq1$, in which case both sides of (\ref{eq:count}) are zero. In particular, Conjecture \ref{cj:OI}(3) holds.
\end{lemma}
\begin{proof}
The situation $\eta_{E/F}(\Delta_{a,b})\neq1$ happens if and only if $E/F$ is nonsplit and $\val_F(\Delta_{a,b})=\leng_{\calO_F}(R_a^\vee:R_a)$ is odd. In this case, $\leng_{\calO_E}(R_a^\vee(\calO_E):R_a(\calO_E))$ is odd, therefore there are no self-dual lattices in $R_a(E)$, i.e., $\sNloc_{a,b}=\varnothing$.

Now we show that the LHS of (\ref{eq:count}) is also zero. For an $\calO_F$-lattice $\Lambda\in R_a(F)$, the linear dual $\Lambda^\vee=\Hom_{\calO_F}(\Lambda,\calO_F)$ can be naturally viewed as another $\calO_F$-lattice of $R_a(F)$ via the identification $\gamma_{a,b}:R_a(F)\isom R_a^\vee(F)$. It is easy to check that if $\Lambda$ is stable under multiplication by $R_a$,  the same is true for $\Lambda^\vee$. This operation sets up a bijection
\begin{equation*}
(-)^\vee:\sMloc_{i,a,b}\isom\sMloc_{\val_F(\Delta_{a,b})-i,a,b}.
\end{equation*}
Since $E/F$ is nonsplit and $\val_F(\Delta_{a,b})$ is odd, then by Proposition \ref{p:GLloc}
\begin{equation*}
\sum_i(-1)^i\#\sMloc_{i,a,b}=\sum_{i=0}^{\lfloor\val_F(\Delta_{a,b})/2\rfloor}(-1)^i(\#\sMloc_{i,a,b}-\#\sMloc_{\val_F(\Delta_{a,b})-i,a,b})=0.
\end{equation*}
This completes the proof.
\end{proof}

\begin{remark}
In \cite{JR}, Jacquet-Rallis showed that every strongly regular semisimple $A'\in\fru_n(F)$ matches some $A\in\frs_n(F)$; conversely, a strongly regular semisimple $A\in\frs_n(F)$ matches some $A'\in\fru_n(F)$ if and only if $\eta_{E/F}(\Delta_{a,b})=1$, where $(a,b)$ are the invariants of $A$. We have seen from Lemma \ref{l:cancel} that if $A\in\frs_n(F)$ does not match any element in $\fru_n(F)$, then $\OI^{\GL_{n-1},\eta}_A(\one_{\frs_n}(\calO_F))=0$. This vanishing result was also conjectured in \cite{JR}.
\end{remark}

\begin{lemma}\label{l:split}
The Conjecture \ref{cj:count} is true if $E/F$ is split.
\end{lemma}
\begin{proof}
Fix an isomorphism $E\cong F\oplus F$ such that $\sigma$ interchanges the two factors. Using this, we can identify $R_a(E)$ is with $R_a(F)\oplus R_a(F)$ and the Hermitian form $(\cdot,\cdot)_R$ takes the form:
\begin{equation*}
(x\oplus y,x'\oplus y')_R=b(xy')\oplus b(x'y)\in E,\hspace{1cm}x,y,x',y'\in R_a(F).
\end{equation*}
Therefore each $R_a(\calO_E)$-lattice $R_a(\calO_E)\subset\Lambda'\subset R_a^\vee(\calO_E)$ has the form $\Lambda'=\Lambda_1\oplus\Lambda_2$, with $R_a\subset\Lambda_i\subset R_a^\vee$. The self-duality requirement is equivalent to $\Lambda_2=\gamma_{a,b}^{-1}(\Lambda_1^\vee)$ (note that $\Lambda_1^\vee\subset R_a^\vee(F)$ and recall the isomorphism $\gamma_{a,b}:R_a(F)\isom R_a^\vee(F)$). In this way, we get a bijection
\begin{eqnarray*}
\coprod_{i=0}^{\val_F(\Delta_{a,b})}\sMloc_{i,a,b}&\isom&\sNloc_{a,b}\\
\Lambda&\mapsto&\Lambda\oplus\gamma_{a,b}^{-1}(\Lambda^\vee).
\end{eqnarray*}
Therefore,
\begin{equation*}
\sum_{i}\#\sMloc_{i,a,b}=\#\sNloc_{a,b},
\end{equation*}
which verifies Conjecture \ref{cj:count} in the split case because $\eta_{E/F}$ is trivial.
\end{proof}

\begin{lemma}\label{l:DVR}
The Conjecture \ref{cj:count} is true if $R_a$ is a product of DVRs.
\end{lemma}
\begin{proof}
By reducing to the lattice-counting problems, it is clear that it suffices to deal with the case $R_a$ is a DVR. Since we already dealt with the split case, we may assume $E/F$ is nonsplit. Let $k(R_a)$ be the residue field of $R_a$ and $\varpi_R$ be a uniformizing parameter of $R_a$. Then $R_a^\vee=\varpi_R^{-d}R_a$ for some integer $d=\leng_{R_a}(R^\vee_a:R_a)$. We have
\begin{equation}\label{eq:leng}
\val_F(\Delta_{a,b})=\leng_{\calO_F}(R^\vee_a:R_a)=d[k(R_a):k]
\end{equation}
We already solve the case when $\val_F(\Delta_{a,b})$ is odd in Lemma \ref{l:cancel} and when $E/F$ is split in Lemma \ref{l:split}. Now suppose $\val_F(\Delta_{a,b})$ is even and $E/F$ is nonsplit. In this case, either $d$ or $[k(R_a):k]$ has to be even. Let us explicitly count the cardinalities of $\sMloc_{i,a,b}$ and $\sNloc_{a,b}$. 

On one hand, the only $R_a$-lattices which sit between $R_a$ and $R_a^\vee$ are $\varpi_R^{-j}R_a$ for $0\leq j\leq d$, and $\leng_{\calO_F}(R_a^\vee:\varpi_R^{-j}R_a)=(d-j)[k(R_a):k]$. Therefore
\begin{equation*}
\sum_i(-1)^i\#\sMloc_{i,a,b}=\sum_{j=0}^{d}(-1)^{(d-j)[k(R_a):k]}=\left\{
\begin{array}{ll}
d+1 & [k(R_a):k]\textup{ even;}\\ 
1 & [k(R_a):k]\textup{ odd.}
\end{array}\right.
\end{equation*}

On the other hand, if $[k(R_a):k]$ is odd and $d$ is even, $R_a(\calO_E)$ remains a DVR, therefore $\varpi_R^{-d/2}R_a(\calO_E)$ is the unique self-dual $R_a(\calO_E)$ lattice between $R_a(\calO_E)$ and $R_a^\vee(\calO_E)=\varpi_R^{-d}R_a(\calO_E)$. If $[k(R_a):k]$ is even, then we can identify $R_a(\calO_E)\cong R_a\oplus R_a$ such that $\sigma_R$ acts by interchanging the two factors. In this case $\sNloc_{a,b}$ consists of lattices $\varpi_R^{-j}R_a\oplus\varpi_R^{-d+j}R_a$ for $0\leq j\leq d$. In any case, we have
\begin{equation*}
\sum_{i}(-1)^i\#\sMloc_{i,a,b}=\#\sNloc_{a,b}.
\end{equation*}
\end{proof}

\subsection{From the Lie algebra version to the group version}\label{ss:gptoalg}
As mentioned in the Introduction, it is the group version identity \eqref{eq:gpOI} which is directly relevant to Jacquet-Rallis's approach to the Gross-Prasad conjecture for the unitary groups. In this subsection we deduce the group version \eqref{eq:gpOI} from the Lie algebra version \eqref{eq:OI}. The same argument also shows that \eqref{eq:gpOI'} follows from \eqref{eq:algOI'}. 

For an element $A\in\GL_n(E)$, viewed as an element in $\frgl_n(E)$, the invariants $a_i(A),b_i(A)$ and $v(A)$ are defined as in \S \ref{ss:inv}. When $a_i,b_i\in\calO_E$, we introduce the $\calO_E$-algebra
\begin{equation}
\bR_a(\calO_E)=\calO_E[t,t^{-1}]/(t^n-a_1t^{n-1}+\cdots+(-1)^na_n).
\end{equation}
We view $Z'_a=\Spec\bR_a(\calO_E)$ as a subscheme of $\Spec\calO_E\times\GG_m$ which is finite flat over $\Spec\calO_E$ of degree $n$. Let $\theta$ be the involution on $\Spec\calO_E\times\GG_m$ which is the product of $\sigma$ on $\calO_E$ and $t\mapsto t^{-1}$ on $\GG_m$. The fixed point subcsheme under $\theta$ is the unitary group $\Ug_{\calO_E/\calO_F}(1)$ under $\Spec\calO_F$.

Recall that $\Sg_n(\calO_F)=\{A\in\GL_n(\calO_E)|A\sigma(A)=1\}$. For an element $A$ in either $S_n(\calO_F)$ or $\Ug_n(\calO_F)$, the subscheme $Z'_a$ is stable under $\theta$, hence determining a subscheme $Z_a\subset\Ug_{\calO_E/\calO_F}(1)$, finite flat of degree $n$ over $\Spec\calO_F$. Let $\bR_a$ be the coordinate ring of $Z_a$, which is a finite flat $\calO_F$-algebra of rank $n$ satisfying $\bR_a\otimes_{\calO_F}\calO_E=\bR_a(\calO_E)$. The invariants $b_i$ determines an $\bR_a$-linear map $\gamma_{a,b}:\bR_a\to\bR_a^\vee$, as in \eqref{eq:gamma}.

\begin{prop}\label{p:gptoalg}
Conjecture \ref{cj:OI}(2) follows from Conjecture \ref{cj:OI}(1).
\end{prop}
\begin{proof}
Using the same argument for Propositions \ref{p:GLloc} and \ref{p:Uloc}, we reduce the orbital integrals in \eqref{eq:gpOI} to counting of points in the corresponding sets $\bM^{\textup{loc}}_{i,a,b}$ and $\bN^{\textup{loc}}_{a,b}$, defined using $\bR_a$ instead of $R_a$. Therefore, it suffices to find $\tila_i,\tilb_i\in\calO_E^{\sigma=(-1)^i}$ and an isomorphism of $\calO_F$-algebras $\rho:R_{\tila}\isom\bR_a$ such that the following diagram is commutative:
\begin{equation}\label{eq:gammacomm}
\xymatrix{R_{\tila}\ar[r]^{\gamma_{\tila,\tilb}}\ar[d]^{\rho}_{\wr} & R_{\tila}^\vee\\
\bR_a\ar[r]^{\gamma_{a,b}} & \bR^\vee_a\ar[u]_{\rho^\vee}^{\wr}}
\end{equation}
Moreover, once we find $\rho:R_{\tila}\isom\bR_{a}$, the choice of $\tilb$ is uniquely determined by the diagram \eqref{eq:gammacomm}, because the data of $\tilb$ and $\gamma_{\tila,\tilb}$ determine each other, as seen in \eqref{eq:bprime} and \eqref{eq:pairing}. Therefore we only need to find $\tila$ such that $\bR_a$ is isomorphic to $R_{\tila}$.

Consider the special fiber of the $\calO_F$-scheme $Z_a$, which is a finite subscheme of $\Ug_{k'/k}(1)$ of degree $n$. Since $\Ug_{k'/k}(1)$ is a smooth curve, any subscheme of it can be embedded into $\AA^1_{k}$. In other words, there is a surjection of algebras $k[s]\twoheadrightarrow\bR_a\otimes_{\calO_F}k$. Lifting the image of the generator $s$ to an element of $\bR_a$, we get a surjection of $\calO_F$-algebras $\calO_F[s]\twoheadrightarrow\bR_a$ (surjectivity follows from Nakayama's lemma). In other words, $Z_a$ can be embedded as a subscheme of $\Spec\calO_F\times\AA^1$. It is well-known that any such finite flat $\calO_F$-subscheme of $\Spec\calO_F\times\AA^1$ is defined by one equation of the form $t^n-c_1t^{n-1}+\cdots+(-1)^{n}c_n$ for some $c_i\in\calO_F$, i.e., there is an isomorphism
\begin{equation}\label{RFc}
R^F_c:=\calO_F[s]/(t^n-c_1t^{n-1}+\cdots+(-1)^{n}c_n)\isom\bR_a.
\end{equation}
Let $\jmath\in \calO_E^-\cap\calO_E^\times$ be a purely imaginary unit element and let $\tila_i=\jmath^ic_i\in\calO_E^{\sigma=(-1)^{i}}$, then we have an isomorphism $R_{\tila}\isom R^F_c$ by sending $t\mapsto\jmath^{-1}s$. Composing with \eqref{RFc}, we get the desired isomorphism of $\calO_F$-algebras $\rho:R_{\tila}\isom\bR_a$. This completes the proof.
\end{proof}

\subsection{Geometric reformulation}\label{ss:locmatch}
In this subsection, we assume $\char(F)=\char(k)$. In this case, we can interpret the sets $\sMloc_{i,a,b}$ and $\sNloc_{a,b}$ as $k$-points of certain schemes.

Fix a strongly regular semisimple pair $(a,b)$ such that $a_i,b_i\in\calO_E^{\sigma=(-1)^i}$. For $0\leq i\leq\val_F(\Delta_{a,b})$, consider the following functor
\begin{equation*}
S\mapsto\left\{
\begin{array}{l|l}
R_{a}\boxtimes_{k}\calO_S-& R_{a}\boxtimes_{k}\calO_S\subset\Lambda\subset R_{a}^\vee\boxtimes_{k}\calO_S,\\
\textup{modules }\Lambda_x & R_{a}^\vee\boxtimes_{k}\calO_S/\Lambda\textup{ is a vector bundle of rank }i\textup{ over }S
\end{array}\right\}.
\end{equation*}
It is clear that this functor is represented by a projective scheme $\Mloc_{i,a,b}$ over $k$, and
\begin{equation*}
\sMloc_{i,a,b}=\Mloc_{i,a,b}(k).
\end{equation*}

Similarly, we have a projective scheme $\Nloc_{a,b}$ over $k$ representing the functor
\begin{equation*}
S\mapsto\left\{
\begin{array}{l|l}
\textup{Self-dual }R_{a}(\calO_x)\boxtimes_{k} & R_{a}(\calO_E)\boxtimes_{k}\calO_S\subset\Lambda'\subset R_{a}^\vee(\calO_E)\boxtimes_{k}\calO_S,\\
\calO_S\textup{-modules }\Lambda'& R_{a}^\vee(\calO_E)\boxtimes_{k}\calO_S/\Lambda'\textup{ is a vector bundle over }\calO_S
\end{array}\right\}
\end{equation*}
We also have
\begin{equation*}
\sNloc_{a,b}=\Nloc_{a,b}(k).
\end{equation*}

Let $\ell$ be a prime number different from $\char(k)$. Let $\Ql(\eta_{k'/k})$ be the rank one $\Ql$-local system on $\Spec k$ associated to the extension $k'/k$: it is trivial if $E/F$ is split and has order two otherwise.

The local part of the main theorem of the paper is:
\begin{theorem}\label{th:main} Suppose $\char(F)=\char(k)>\max\{n,2\}$ and $\eta_{E/F}(\Delta_{a,b})=1$. Then there is an isomorphism of graded $\Frob_k$-modules:
\begin{equation}\label{eq:main}
\bigoplus_{i=0}^{\val_F(\Delta_{a,b})}H^*(\geom{\Mloc_{i,a,b}}{k},\Ql(\eta_{k'/k})^{\otimes i})\cong H^*(\geom{\Nloc_{a,b}}{k},\Ql)
\end{equation}
\end{theorem}

Taking Frobenius traces of the isomorphism (\ref{eq:main}), we get
\begin{cor}\label{c:main}
Conjecture \ref{cj:count}, hence Conjecture \ref{cj:OI} is true if $\char(F)=\char(k)>n$.
\end{cor}

The proof of Theorem \ref{th:main} will be completed in \S \ref{s:pf}. We first prove an easy case.

\begin{lemma}\label{l:locisom}
Let $\Omega=k$ if $k'/k$ is split, or $\Omega=k'$ if $k'/k$ is nonsplit, we have an isomorphism of schemes over $\Omega$:
\begin{equation}\label{eq:NM}
\coprod_{i=0}^{\val_F(\Delta_{a,b})}\Mloc_{i,a,b}\otimes_k\Omega\isom\Nloc_{a,b}\otimes_k\Omega,
\end{equation}
\end{lemma}
\begin{proof}
After base change to $\Omega$, we may assume that $E/F$ is split. Then the argument is the same as the proof of Lemma \ref{l:split}, once we fix an identification $R_a(E)\cong R_a(F)\oplus R_a(F)$.
\end{proof}

\begin{cor}
Theorem \ref{th:main} holds if $E/F$ is split.
\end{cor}

\section{Global formulation---the moduli spaces}\label{s:mod}
Let $k=\FF_q$ be a finite field with $\char(k)>\max\{n,2\}$. Let $X$ be a smooth, projective and geometrically connected curve over $k$ of genus $g$. Let $\pi:X'\to X$ be an \'etale double cover such that $X'/k$ is also geometrically connected. Let $\sigma$ denote the nontrivial involution of $X'$ over $X$. We have a canonical decomposition:
\begin{equation*}
\pi_*\calO_{X'}=\calO_X\oplus\calL
\end{equation*}
into $\pm1$-eigenspaces of $\sigma$. Here $\calL$ is a line bundle on $X$ such that $\calL^{\otimes2}\cong\calO_X$.

Let $D$ and $D_0$ be effective divisors on $X$. Assume $\deg(D)\geq2g$.

\subsection{The moduli spaces associated to $\frs_n$}\label{ss:mods}
Consider the functor $\underline{\calM}:\textup{Sch}/k\to\textup{Grpd}$:
\begin{eqnarray*}
S\mapsto\left\{
\begin{array}{l|l}
&\calE\textup{ is a vector bundle of rank }n\textup{ over }X\times S,\\
(\calE,\phi,\lambda,\mu)&\phi:\calE\to\calE\otimes_{\calO_X}\calL(D),\\
&\calO_{X\times S}(-D_0)\xrightarrow{\lambda}\calE\xrightarrow{\mu}\calO_{X\times S}(D_0)
\end{array}\right\}.
\end{eqnarray*}
Here, the twisting by $(D)$ or $(D_0)$ means tensoring with the pull-back of the line bundles $\calO_X(D)$ or $\calO_X(D_0)$ to $X\times S$. For each integer $i$, we define the subfunctor $\underline{\calM}_i$ of $\underline{\calM}$ by taking only those vector bundles $\calE$ such that
\begin{equation*}
\chi(X\otimes_kk(s),\calE\otimes_kk(s))=i-n(g-1)
\end{equation*}
for any geometric point $s$ of $S$. It is clear that $\underline{\calM}_i$ is represented by an algebraic stack $\calM_i$ over $k$ locally of finite type; $\underline{\calM}$ is represented by $\calM=\coprod_i\calM_i$.

Let $\MHit_i$ be the {\em Hitchin moduli stack} for $\GL_n$ (with the choice of the line bundle $\calL(D)$ on $X$) which classifies only the pairs $(\calE,\phi)$ as above. For more details about this Hitchin stack, we refer the readers to \cite[\S 4.2, \S 4.7]{NFL}.

Let
\begin{eqnarray*}
\calA&:=&\bigoplus_{i=1}^{n}H^0(X,\calL(D)^{\otimes i});\\
\calB&:=&\bigoplus_{i=0}^{n-1}H^0(X,\calO_X(2D_0)\otimes\calL(D)^{\otimes i}).
\end{eqnarray*}
viewed as affine spaces over $k$. We have a natural morphism:
\begin{equation*}
f_i:\calM_i\to\calA\times\calB
\end{equation*}
which, on the level of $S$-points sends $(\calE,\phi,\alpha_X,\beta_X)$ to $a=(a_1,\cdots,a_n)\in\calA$ and $b=(b_0,\cdots,b_{n-1})\in\calB$ where
\begin{equation*}
a_i=\Tr(\bigwedge^i\phi)\in H^0(X\times S,\calL(D)^{\otimes i})\\
\end{equation*}
and $b_i\in H^0(X\times S,\calO_{X\times S}(2D_0)\otimes\calL(D)^{\otimes i})=\Hom_{X\times S}(\calO_{X\times S}(-D_0),\calO_{X\times S}(D_0)\otimes\calL(D)^{\otimes i})$ is represented by the following homomorphism
\begin{equation*}
\calO_{X\times S}(-D_0)\xrightarrow{\lambda}\calE\xrightarrow{\phi^i}\calE\otimes\calL(D)^{\otimes i}\xrightarrow{\mu}\calO_{X\times S}(D_0)\otimes\calL(D)^{\otimes i}
\end{equation*}

\subsection{The spectral curves}
Following \cite[\S 2.5]{LN}, we define the universal spectral curve $p:Y\to\calA\times X$ as follows. For each $S$-point $a=(a_1,\cdots,a_n)\in\calA(S)$, define the following scheme, affine over $X'\times S$:
\begin{equation*}
Y_a':=\underline{\Spec}_{X'\times S}\left(\bigoplus_{i=0}^{n-1}\calO_{X'\times S}(-iD)t^i\right)
\end{equation*}
where the ring structure on the RHS is defined by the relation
\begin{equation*}
t^n-a_1t^{n-1}+a_2t^{n-2}-\cdots+(-1)^na_n=0.
\end{equation*}
Let $p_a':Y_a'\to X'\times S$ be the natural projection. This is a finite flat morphism of degree $n$. The scheme $Y_a'$ over $X'\times S$ naturally embeds into the total space $\Tot_{X'\times S}(\calO(D))$ of the line bundle $\calO_{X'\times S}(D)$ over $X'\times S$. The free involution $\sigma$ on $X'$ extends to a free involution on $Y_a'$ by requiring $\sigma(t)=-t$. The quotient of $Y_a'$ by $\sigma$ is the scheme
\begin{equation*}
Y_a:=\underline{\Spec}_{X\times S}\calO_{Y_a'}^{\sigma}=\underline{\Spec}_{X\times S}\left(\bigoplus_{i=0}^{n-1}\calL(-D)^{\otimes i}\boxtimes\calO_{S}t^i\right).
\end{equation*}
Let $p_a:Y\to X\times S$ be the natural projection. This is a finite flat morphism of degree $n$. The scheme $Y_a$ naturally embeds into the total space $\Tot_{X\times S}(\calL(D))$ of the line bundle $\calL(D)$ over $X\times S$. The quotient map $\pi_a:Y_a'\to Y_a$ is an \'etale double cover.

Let $\Ah$ (resp. $\Ared$, resp. $\Asm$) be the open subset of $\calA$ consisting of those geometric points $a$ such that $Y_a'$, and hence $Y_a$, are integral (resp. reduced, resp. smooth and irreducible). Let $\Bh=\calB-\{0\}$. Let $\Mh_i$ be the restriction of $\calM_i$ to $\Ah\times\Bh$.

\begin{lemma}\label{l:intcomp}
The codimension of $\Ared-\Ah$ in $\Ared$ has codimension at least $\deg(D)$.
\end{lemma}
\begin{proof}
In our situation, $\calA$ serves as the Hitchin base for $\GL_n$ and $\Ug_n$ at the same time (the unitary Hitchin stack will be recalled in \S \ref{ss:modU}). The locus $\Ah$ is in fact the intersection of two elliptic loci: $\Ah=\calA^{\ellp}_{\GL_n}\cap\calA^{\ellp}_{\Ug_n}$. Here, $\calA^{\ellp}_{\GL_n}$ is the locus where $Y_a$ is irreducible; $\calA^{\ellp}_{\Ug_n}$ is the locus where the set of irreducible components of $Y_a'$ is in bijection with that of $Y_a$ (cf. \cite[\S 2.8]{LN}. In the $\Ug_n$ case, the elliptic locus $\calA^{\ellp}_{\Ug_n}$ is the same as the anisotropic locus considered in \cite{NFL}, and by Proposition 6.5.1 of {\em loc.cit.}, we have
\begin{equation*}
\codim_{\Ared}(\Ared-\calA^{\ellp}_{\Ug_n})\geq\deg(D).
\end{equation*}
In the $\GL_n$ case, the same argument also works to prove that
\begin{equation*}
\codim_{\Ared}(\Ared-\calA^{\ellp}_{\GL_n})\geq\deg(D).
\end{equation*}
In fact, we only need to compute the dimension of the Hitchin bases for the Levi subgroups $\GL_{n_1}\times\cdots\times\GL_{n_r}$.

Therefore,
\begin{equation*}
\dim(\Ared-\Ah)=\dim\left((\Ared-\calA^{\ellp}_{\Ug_n})\cup(\Ared-\calA^{\ellp}_{\GL_n})\right)\leq\dim\Ared-\deg(D).
\end{equation*}
\end{proof}

The following lemma is a direct calculation.
\begin{lemma}\label{l:arithg}
For a geometric point $a\in\Ah$, the arithmetic genera of the curves $Y_a'$ and $Y_a$ are
\begin{eqnarray*}
g'_Y:=1-\chi(Y_a',\calO_{Y_a'})&=&n(n-1)\deg(D)+(2g-2)n+1;\\
g_Y:=1-\chi(Y_a,\calO_{Y_a})&=&n(n-1)\deg(D)/2+(g-1)n+1.
\end{eqnarray*}
\end{lemma}

Recall that for a locally projective flat family of geometrically integral curves $C$ over $S$, we have the compactified Picard stack $\cPic(C/S)=\coprod_i\cPic^i(C/S)$ over $S$ (see \cite{AK}) whose fiber over a geometric point $s\in S$ classifies the groupoid of torsion-free coherent sheaves $\calF$ of generic rank 1 over $C_s$ such that $\chi(C_s,\calF)=i$. Each $\cPic^i(C/S)$ is an algebraic stack of finite type over $S$; it is in fact a $\GG_m$-gerb over the compactified Picard scheme of $\overline{\Pic}^i(C/S)$. The scheme $\overline{\Pic}^i(C/S)$ is proper over $S$ and contains the usual Picard scheme $\Pic^i(C/S)$ as an open substack.

For each $a\in\Ah(S)$ and $\calF\in\cPic(Y_a/S)$, the coherent sheaf $\calE=p_{a,*}\calF$ is a vector bundle of rank $n$ over $X\times S$ which is naturally equipped with a Higgs field $\phi:\calE\to\calE\otimes_{\calO_X}\calL(D)$; conversely, every object $(E,\phi)\in\MHit(S)$ over $a\in\Ah(S)$ comes in this way. Therefore we have a natural isomorphism of stacks (cf. \cite{BNR})
\begin{equation}\label{eq:HitPic}
\cPic(Y/\Ah)\isom\MHit|_{\Ah}.
\end{equation}
Therefore we can view $\Mh_i$ as a stack over $\cPic(Y/\Ah)$.

\begin{lemma}\label{l:modpic}
The stack $\Mh_i$ represents the following functor
\begin{equation*}
S\mapsto\left\{
\begin{array}{l|l}
&a\in\Ah(S),\calF\in\cPic^{i-n(g-1)}(Y_a/S),\\
(a,\calF,\alpha,\beta)&\calO_{Y_a}(-D_0)\xrightarrow{\alpha}\calF\xrightarrow{\beta}\omega_{Y_a/X\times S}(D_0)\textup{ such that }\gamma=\beta\circ\alpha\textup{ is}\\
&\textup{ nonzero along each geometric fiber of }Y_a\to S
\end{array}
\right\}.
\end{equation*}
\end{lemma}
\begin{proof}
For a quadruple $(a,\calF,\alpha,\beta)$ as above, we associate $(\calE=p_{a,*}\calF,\phi)\in\MHit_i$ by the isomorphism (\ref{eq:HitPic}). By adjunction, we have
\begin{eqnarray*}
\Hom_X(\calO_{X\times S}(-D_0),\calE)=\Hom_{Y_a}(\calO_{Y_a}(-D_0),\calF)\\
\Hom_X(\calE,\calO_{X\times S}(D_0))=\Hom_{Y_a}(\calF,\omega_{Y_a/X\times S}(D_0)).
\end{eqnarray*}
Therefore from $(\alpha,\beta)$ we can associate a unique pair of homomorphisms
\begin{equation*}
\calO_{X\times S}(-D_0)\xrightarrow{\lambda}\calE\xrightarrow{\mu}\calO_{X\times S}(D_0).
\end{equation*}
Let $b=(b_0,\cdots,b_{n-1})\in\calB$ be the second collection of invariants of $(\calE,\phi,\lambda,\mu)$. The composition $\gamma=\beta\circ\alpha$ is an element in
\begin{equation*}
\Hom_{Y_a}(\calO_{Y_a}(-D_0),\omega_{Y_a/X\times S}(D_0))=\Hom_{X\times S}(\calO_{Y_a},\calO_{X\times S}(2D_0))
\end{equation*}
which is given by
\begin{equation*}
(b_0,\cdots,b_{n-1}):\calO_{Y_a}=\bigoplus_{i=0}^{n-1}\calL(-D)^{\otimes i}\to\calO_{X\times S}(2D_0).
\end{equation*}
Therefore for any geometric point $s$ of $S$, the condition $b(s)\neq0$ is equivalent to that $\gamma|_{Y_s}\neq0$.
\end{proof}

\begin{remark}\label{rm:gamma}
From the proof of above lemma, we see that for any $(a,b)\in\Ah\times\Bh$, the homomorphism $\gamma:\calO_{Y_a}(-D_0)\to\omega_{Y_a/X}(D_0)$ is independent of the choice of $(\calF,\alpha,\beta)\in\Mh_{i,a,b}$. We denote this $\gamma$ by $\gamma_{a,b}$. Therefore we get a morphism
\begin{eqnarray*}
\coker(\gamma):\Ah\times\Bh&\to&\Quot^{2d}(\omega(D_0)/Y/\Ah)\\
(a,b)&\mapsto& \coker(\gamma_{a,b}).
\end{eqnarray*}
\end{remark}

\begin{remark}\label{rm:defd}
Let
\begin{equation*}
d=n\deg(D_0)-n(g-1)+g_Y-1=n(n-1)\deg(D)/2+n\deg(D_0)
\end{equation*}
By the moduli interpretation given in Lemma \ref{l:modpic}, $\Mh_i$ is non-empty only if 
\begin{equation*}
-d-n(g-1)=\chi(Y_a,\calO_{Y_a}(-D_0))\leq i-n(g-1)\leq\chi(Y_a,\omega_{Y_a/X}(D_0))=d-n(g-1).
\end{equation*}
(here $a\in\Ah$ is any geometric point); i.e., $-d\leq i\leq d$.
\end{remark}

\begin{prop}\label{p:smGL}
For $-d+2g_Y-1\leq i\leq d-2g_Y+1$, $\Mh_i$ is a scheme smooth over $k$ and the morphism $\fh_i:\Mh_i\to\Ah\times\Bh$ is proper.
\end{prop}
\begin{proof}
We have the following Cartesian diagram
\begin{equation}\label{d:CartM}
\xymatrix{\Mh_i \ar[d]^{r_\alpha} \ar[r]^{r_\beta} & \Quot^{d-i}(\omega(D_0)/Y/\Ah)\ar[d]^{\AJ_{d-i}}\\
\Quot^{d+i}(\omega(D_0)/Y/\Ah)\ar[r]^(.6){(\AJ_{d+i})^\vee} & \cPic^{i-n(g-1)}(Y/\Ah)}
\end{equation}
which, on the level of $S$-points over $a\in\Ah(S)$, are defined as
\begin{equation*}
\xymatrix{(\calF,\alpha,\beta)\ar[r]^{r_\beta}\ar[d]^{r_\alpha} & \coker(\calF\xrightarrow{\beta}\omega_{Y_a/X\times S}(D_0))\ar[d]^{\AJ_{d-i}}\\
\coker(\omega_{Y_a/X\times S}(D_0)\xrightarrow{\alpha^\vee}\calF^\vee)\ar[r]^(.6){(\AJ_{d+i})^\vee} & \calF}
\end{equation*}
The only thing that we need to check is that $\coker(\beta)$ and $\coker(\alpha^\vee)$ are finite {\em flat} $\calO_S$-modules of rank $d-i$ and $d+i$ respectively (when $S$ is locally noetherian). We check this for $\calQ:=\coker(\beta)$. For any geometric point $s\in S$, the map $\gamma_s:\calO_{Y_{a,s}}(-D_0)\xrightarrow{\alpha_s}\calF_s\xrightarrow{\beta_s}\omega_{Y_{a,s}/X_s}(D_0)$ is nonzero, hence generically an isomorphism. Therefore $\beta_s$ is surjective on the generic point of $Y_{a,s}$. Since $\calF_s$ is torsion-free of rank 1, we conclude that $\beta_s$ is injective. Since $\omega_{Y_a/X\times S}(D_0)$ is flat over $\calO_S$, we have
\begin{equation*}
\Tor_1^{\calO_S}(\calQ,k(s))=\ker(\beta_s)=0.
\end{equation*}
This being true for any geometric point $s\in S$, we conclude that $\calQ$ is flat over $\calO_S$. The rank of $\calQ_s$ over $k(s)$ for any geometric point $s\in S$ is
\begin{equation*}
\chi(Y_{a,s},\omega_{Y_{a,s}/X_s}(D_0))-\chi(Y_{a,s},\calF_s)=d-i.
\end{equation*}

Since the Abel-Jacobi map $\AJ_{d-i}$ is schematic and $\Quot^{d+i}(\omega(D_0)/Y/\Ah)$ is a scheme, $\Mh_i$ is also a scheme. For $i$ in the given range, we have $d\pm i\geq 2g_Y-1$, therefore $(\AJ_{d+i})^\vee$ and $\AJ_{d-i}$ are smooth by \cite[Theorem 8.4(v)]{AK}. Moreover, since we assumed $\deg(D)\geq2g$ from the beginning, $\cPic(Y/\Ah)\cong\MHit|_{\Ah}$ is smooth by \cite[Proposition 4.12.1]{NFL}. Therefore, $\Mh_i$ is a smooth scheme over $k$.

We have a $\GG_m\times\GG_m$ actions on $\Mh_i$: $(c_1,c_2)$ acts by changing $(\calF,\alpha,\beta)$ to $(\calF,c_1\alpha,c_2\beta)$. It is easy to see that $(c,c^{-1})$ acts trivially so that the action factors through the multiplication map $\GG_m\times\GG_m\to\GG_m$. This $\GG_m$-action is free and the quotient $\bMh_i$ exists as a scheme. In fact, we can define $\bMh_i$ by a similar Cartesian diagram as (\ref{d:CartM}) (the only difference is that the stack $\cPic^{i-n(g-1)}(Y/\Ah)$ is replaced by the scheme $\overline{\Pic}^{i-n(g-1)}(Y/\Ah)$):
\begin{equation}\label{d:CartbM}
\xymatrix{\bMh_i \ar[d]\ar[r] & \Quot^{d-i}(\omega(D_0)/Y/\Ah)\ar[d]^{\AJ_{d-i}}\\
\Quot^{d+i}(\omega(D_0)/Y/\Ah)\ar[r]^(.6){(\AJ_{d+i})^\vee} & \overline{\Pic}^{i-n(g-1)}(Y/\Ah)}.
\end{equation}
In other words, $\bMh_i$ classifies isomorphism classes of $(\calF,\alpha,\beta)$ up to rescaling $\alpha$ and $\beta$. It is clear that $\bMh_i$ is the quotient of $\Mh_i$ under the above-mentioned $\GG_m$-action so that $\Mh_i$ becomes a $\GG_m$-torsor over $\bMh_i$.

From diagram (\ref{d:CartbM}), we see that $\bMh_i$ is proper over $\Ah$ because the Quot-scheme $\Quot^{d+i}(\omega(D_0)/Y/\Ah)$ is proper over $\Ah$ and the morphism $\AJ_{d-i}$ is proper (with fibers isomorphic to projective spaces). Moreover we have a Cartesian diagram
\begin{equation*}
\xymatrix{\Mh_i\ar[r]\ar[d]^{\fh_i} & \bMh_i\ar[d]^{\overline{\fh_i}}\\
\Ah\times\Bh\ar[r] & \Ah\times\PP\Bh}
\end{equation*}
where the horizontal maps are $\GG_m$ torsors. Since both $\bMh_i$ and $\Ah\times\PP\Bh$ are proper over $\Ah$, the morphism $\overline{\fh_i}$ is proper. Therefore $\fh_i$ is also proper. This completes the proof.
\end{proof}

\subsection{The moduli space associated to $\fru_n$}\label{ss:modU}
Consider the functor $\underline{\calN}:\textup{Sch}/k\to\textup{Grpd}$:
\begin{eqnarray*}
S\mapsto\left\{
\begin{array}{l|l}
&\calE'\textup{ is a vector bundle of rank }n\textup{ over }X'\times S,\\
(\calE',h,\phi',\mu')&h:\calE'\isom\sigma^*(\calE')^\vee\textup{ is a Hermitian form, i.e., }\sigma^*h^\vee=h,\\
&\phi':\calE'\to\calE'(D)\textup{ such that }\sigma^*\phi'^\vee\circ h+h\circ\phi'=0,\\
&\mu':\calE'\to\calO_{X'\times S}(D_0)
\end{array}\right\}.
\end{eqnarray*}
Here $(-)^\vee=\underline{\Hom}_{X}(-,\calO_X)$. It is clear that $\underline{\calN}$ is represented by an algebraic stack $\calN$ locally of finite type.

Recall that we also have the usual Hitchin moduli stack $\NHit$ for $\Ug_n$ classifying only the triples $(\calE',h,\phi')$ as above. For details about this Hitchin stack, we refer the readers to \cite{LN}.

For $(\calE',h,\phi',\mu')\in\calN(S)$, since $\sigma^*\phi'^\vee=-h\circ\phi'\circ h^{-1}$, we have
\begin{equation*}
a_i=\Tr(\bigwedge^i\phi')\in H^0(X'\times S,\calO_{X'\times S}(iD))^{\sigma=(-1)^i}=H^0(X,\calL(D)^{\otimes i}).
\end{equation*}
Let $\lambda'=h^{-1}\circ\sigma^*\mu'^\vee:\calO_{X'\times S}(-D_0)\to\sigma^*\calE'^\vee\isom\calE'$. Consider the homomorphism
\begin{equation*}
b_i':\calO_{X'\times S}(-D_0)\xrightarrow{\lambda'}\calE\xrightarrow{\phi'^i}\calE'(iD)\xrightarrow{\mu'}\calO_{X'\times S}(D_0+iD).
\end{equation*}
We have a canonical isomorphism
\begin{eqnarray*}
\iota:\sigma^*b_i'^\vee&\cong&\sigma^*\lambda'^\vee\circ(\sigma^*\phi'^\vee)^i\circ\sigma^*h^\vee\circ\sigma^*\mu'^\vee\\
&\cong&\mu'\circ h^{-1}\circ(-h\circ\phi'\circ h^{-1})^i\circ h\circ\lambda'\\
&=&(-1)^i\mu'\circ\phi'^i\circ\lambda'=(-1)^ib_i'.
\end{eqnarray*}
such that $\sigma^*\iota^\vee=\iota$. Therefore $b_i'$ comes from a homomorphism
\begin{equation*}
b_i:\calO_{X\times S}(-D_0)\to\calO_{X\times S}(D_0)\otimes\calL(D)^{\otimes i}.
\end{equation*}
In other words, we may view $b_i$ as an element in $H^0(X\times S,\calO_{X\times S}(2D_0)\otimes\calL(D)^{\otimes i})$. The map that sends $(\calE',h,\phi',\lambda')\in\calN(S)$ to $a=(a_1,\cdots,a_n)\in\calA(S)$ and $b=(b_0,\cdots,b_{n-1})\in\calB(S)$ defines a morphism
\begin{equation*}
g:\calN\to\calA\times\calB.
\end{equation*}

Let $\Nh$ be the restriction of $\calN$ to $\Ah\times\Bh$. Similar to Lemma \ref{l:modpic}, we can rewrite $\Nh$ in terms of spectral curves.

\begin{lemma}\label{l:modpicU}
The stack $\Nh$ represents the following functor
\begin{equation*}
S\mapsto\left\{
\begin{array}{l|l}
&a\in\Ah(S),\calF'\in\cPic(Y_a'/S),h:\calF'\isom\sigma^*(\calF'^\vee),\\
(a,\calF',\beta')&\textup{ such that }\sigma^*h^\vee=h,\textup{ and }\calF'\xrightarrow{\beta'}\omega_{Y_a'/X\times S}(D_0)\\
&\textup{ which is nonzero along each geometric fiber of }Y_a'\to S
\end{array}
\right\}.
\end{equation*}
\end{lemma}

For $(\calF',h,\beta')\in\Nh_{a}(S)$, let $\gamma'=\beta'\circ\alpha'$ where $\alpha'=h^{-1}\circ\sigma^*\beta'^\vee:\calO_{Y_a'}(-D_0)\to\calF'$. Then $\gamma'\in\Hom_{Y_a'}(\calO_{Y_a'}(-D_0),\omega_{Y_a'/X\times S}(D_0))$ satisfies $\sigma^*\gamma'^\vee=\gamma'$. Therefore $\gamma'$ comes from $\gamma\in\Hom_{Y_a}(\calO_{Y_a}(-D_0),\omega_{Y_a/X\times S}(D_0))$ via pull-back along the double cover $\pi_a: Y_a'\to Y_a$. It is easy to see that this $\gamma$ is the same as the $\gamma_{a,b}$ defined in Remark \ref{rm:gamma}. 

\begin{prop}\label{p:smU}
Suppose $n(\deg(D_0)-g+1)\geq g_Y$. Then the stack $\Nh$ is a scheme smooth over $k$ and the morphism $\gh:\Nh\to\Ah\times\Bh$ is proper.
\end{prop}
\begin{proof}
By the moduli interpretation given in Lemma \ref{l:modpicU}, we have a Cartesian diagram
\begin{equation}\label{d:CartN}
\xymatrix{\Nh\ar[r]^(.4){r_\beta'}\ar[d]^{r}&\Quot^{2d}(\omega(D_0)/Y'/\Ah)\ar[d]^{\AJ_{2d}}\\
\NHit|_{\Ah}\ar[r]^{u}&\cPic^{-2n(g-1)}(Y'/\Ah)}
\end{equation}
which, on the level of $S$-points over $a\in\Ah(S)$, are defined as
\begin{equation*}
\xymatrix{(\calF',h,\beta')\ar[r]^(.4){r_\beta'}\ar[d]^{r}&\coker(\calF'\xrightarrow{\beta'}\omega_{Y_a'/X\times S}(D_0))\ar[d]^{\AJ_{2d}}\\
(\calF',h)\ar[r]^{u} &\calF'}
\end{equation*}
Here $\calF'$ has Euler characteristic $-2n(g-1)$ along every geometric fiber of $Y_a'\to S$ because $\calF'\cong\sigma^*\calF'^\vee$.

Since $u$ is schematic and $\Quot^{2d}(\omega(D_0)/Y'/\Ah)$ is a scheme, we see that $\Nh$ is a scheme. By \cite[Proposition 2.5.2, Proposition 2.8.4]{LN}, $\NHit|_{\Ah}$ is smooth over $k$ and proper over $\Ah$. Moreover, since $d=n\deg(D_0-g+1)+g_Y-1\geq2g_Y-1=g_Y'$, $2d\geq2g_Y'$, therefore $\AJ_{2d}$ is smooth by \cite[Theorem 8.4(v)]{AK}. Therefore $\Nh$ is a smooth scheme over $k$.

We have a $\GG_m$-action on $\Nh$ by rescaling $\beta'$. Unlike the case of $\Mh_i$, this action is not free: the subgroup $\mu_2\subset\GG_m$ acts trivially. The quotient of $\Nh$ by this $\GG_m$-action is a Deligne-Mumford stack $\bNh$ proper over $\NHit|_{\Ah}$, hence proper over $\Ah$. Let $(\PP\Bh)'$ be the quotient of $\Bh$ by the square action of the dilation by $\GG_m$. This is a separated Deligne-Mumford stack. Therefore $\bNh$ is proper over $\Ah\times(\PP\Bh)'$. We have a Cartesian diagram
\begin{equation*}
\xymatrix{\Nh\ar[r]\ar[d]^{\gh} & \bNh\ar[d]^{\overline{\gh}}\\
\Ah\times\Bh\ar[r] & \Ah\times(\PP\Bh)'}
\end{equation*}
This implies that $\gh:\Nh\to\Ah\times\Bh$ is also proper.
\end{proof}

\subsection{The product formulae}\label{ss:pd}
We want to express the fibers of $\fh_i$ in terms of local moduli spaces similar to $\Mloc_{i,a,b}$ defined in \S \ref{ss:OIGL}. To simplify the notations, we only consider fibers $\Mh_{i,a,b}$ where $(a,b)\in\Ah(k)\times\Bh(k)$, but the argument is valid for $(a,b)\in\Ah(\Omega)\times\Bh(\Omega)$ for any field $\Omega\supset k$.

Let $|X|$ be the set of closed points of $X$. For any $x\in|X|$, let $\calO_{X,x}$ be the {\em completed} local ring of $X$ at $x$ with fraction field $F_x$ and residue field $k(x)$. Then $\calO_{X,x}$ is naturally a $k(x)$-algebra. Let $\calO_{X',x}$ be the completed semilocal ring of $X'$ along $\pi^{-1}(x)$ (recall Notations \ref{n:semilocal}) and let $E_x$ be its ring of fractions. Then $E_x$ is an unramified (split or nonsplit) quadratic extension of $F_x$. 

Fix an $\sigma$-equivariant trivialization of $\calO_{X'}(D)$ along $\Spec\calO_{X',x}$, which allows us to identify $\calL|_{\Spec\calO_{X,x}}$ with $\calO_{E_x}^{-}$. Fix a trivialization of $\calO_X(D_0)$ over $\Spec\calO_{X,x}$. Using these trivializations, we can identify the restriction of $(a,b)$ on $\Spec\calO_{X,x}$ with a collection of invariants $a_x=(a_{1,x},\cdots,a_{n,x})$ and $b_x=(b_{0,x},\cdots,b_{n-1,x})$ such that $a_{i,x},b_{i,x}\in(\calO_{E_x})^{\sigma=(-1)^i}$. Using $F_x$, $E_x$ and $(a_x, b_x)$ in place of $F$, $E$ and $(a,b)$ in the discussion of \S \ref{ss:inv} and \ref{ss:OIGL}, we can define the $\calO_{F_x}$ algebra $R_{a_x}$, which is isomorphic to $\calO_{Y_a,x}$ (see Notation \ref{n:semilocal}). The trivializations also identify $R_{a_x}^\vee$ with $\omega_{Y_a/X}(D_0)|_{\Spec\calO_{Y_a,x}}$ and $\gamma_{a_x,b_x}:R_{a_x}\to R_{a_x}^\vee$ (defined in (\ref{eq:gamma})) with $\gamma_{a,b}|_{\Spec\calO_{Y,x}}$ (defined in Remark \ref{rm:gamma}). As in (\ref{eq:HermR}), $R_{a_x}(E_x)$ has a natural Hermitian form under which $R_{a_x}(\calO_{E_x})$ and $R_{a_x}^\vee(\calO_{E_x})$ are dual to each other.

With these data, we can define the local moduli spaces of lattice $\Mloc_{i_x,a_x,b_x}$ ($i_x\in\ZZ$) and $\Nloc_{a_x,b_x}$ as in \S \ref{ss:OIGL} and \ref{ss:OIU}, which are projective schemes over $k(x)$. To emphasize its dependence on the point $x$, we denote them by $\calM^x_{i_x,a_x,b_x}$ and $\calN^x_{a_x,b_x}$.

\begin{prop}\label{p:pdGL}
For $(a,b)\in\Ah(k)\times\Bh(k)$, there is an isomorphism of schemes over $k$:
\begin{equation}\label{eq:pdGL}
\calM_{i,a,b}\cong\coprod_{\sum[k(x):k]i_x=d-i}\left(\prod_{x\in|X|}\Res_{k(x)/k}\calM^{x}_{i_x,a_x,b_x}\right).
\end{equation}
Here $|X|$ is the set of closed points of $X$, $\Res_{k(x)/k}$ means restriction of scalars. The disjoint union is taken over the set of all assignments $x\in|X|\mapsto i_x\in\ZZ$ such that $\sum_{x\in|X|}[k(x):k]i_x=d-i$. The product is the fiber product of schemes over $k$.
\end{prop}
\begin{proof}
First we remark that the RHS in the isomorphism (\ref{eq:pdGL}) is in fact finite. Note that for each $x\in|X|$, there are only finitely many $i_x$ such that $\calM^{x}_{i_x,a_x,b_x}$ is non-empty (more precisely, $0\leq i\leq \val_x(\Delta_{a_x,b_x})$). For fixed $(a,b)\in\Ah(k)\times\Bh(k)$, the map $\gamma_{a,b}:\calO_{Y_a}(-D_0)\to\omega_{Y_a/X}(D_0)$ is an isomorphism away from a finite subset $Z\subset|X|$, and when $\gamma_x$ is an isomorphism, $\calM^{x}_{i_x,a_x,b_x}\cong\Spec k(x)$. Therefore we can rewrite the RHS of (\ref{eq:pdGL}) as a finite disjoint union of
\begin{equation*}
\prod_{x\in Z}\Res_{k(x)/k}\calM^{x}_{i_x,a_x,b_x},
\end{equation*}
which makes sense.

By the Cartesian diagram (\ref{d:CartM}), the assignment $(\calF,\alpha,\beta)\mapsto\coker(\beta)$ defines an isomorphism of schemes
\begin{equation*}
\calM_{i,a,b}\cong\Quot^{d-i}(\calQ/Y_a/k).
\end{equation*}
where $\calQ=\coker(\gamma_{a,b})$. Since $\gamma_{a,b}$ is an isomorphism on $X-Z$, $\calQ$ is supported on $Z\times S$. Therefore we get a canonical decomposition
\begin{equation*}
\calQ=\bigoplus_{x\in Z}\calQ_x,
\end{equation*}
with each $\calQ_x$ supported at $x$. We get a corresponding decomposition of the Quot-scheme
\begin{equation*}
\Quot^{d-i}(\calQ/Y_a/k)\cong\coprod_{\sum[k(x):k]i_x=d-i}\left(\prod_{x\in Z}\Res_{k(x)/k}\Quot^{i_x}(\calQ_x/R_{a_x}/k(x))\right).
\end{equation*}
To prove the isomorphism (\ref{eq:pdGL}), it remains to identify $\Quot^{i_x}(\calQ_x/R_{a_x}/k(x))$ with $\calM^{x}_{i_x,a_x,b_x}$, but this is obvious from definition.
\end{proof}

Similarly, we have

\begin{prop}\label{p:pdU}
For $(a,b)\in\Ah(k)\times\Bh(k)$, we have the following isomorphism of schemes over $k$:
\begin{equation}\label{eq:pdU}
\calN_{a,b}\cong\prod_{x\in|X|}\Res_{k(x)/k}\calN^{x}_{a_x,b_x}.
\end{equation}
\end{prop}

\begin{cor}\label{c:globisom}
For any geometric point $(a,b)\in\Ah(\Omega)\times\Bh(\Omega)$, there is an isomorphism of schemes over $\Omega$:
\begin{equation*}
\coprod_{i=-d}^{d}\calM_{i,a,b}\cong\calN_{a,b}.
\end{equation*}
\end{cor}
\begin{proof}
This follows from the two product formulae and Lemma \ref{l:locisom}. 
\end{proof}

\subsection{Smallness}\label{ss:small}

Recall that in \cite{NFL} Ng\^o defines the local and global Serre invariants for points on the Hitchin base. In our case, for a geometric point $a\in\calA(\Omega)$, let $\tilY_a\to Y_a$ be the normalization. Then the local and global Serre invariants are
\begin{eqnarray*}
\delta(a,x)&=&\dim_{\Omega}(\calO_{\tilY_a,x}/\calO_{Y_a,x});\\
\delta(a)&=&\dim_{\Omega}H^0(Y_a,\calO_{\tilY_a}/\calO_{Y_a})=\sum_{x\in X(\Omega)}\delta(a,x).
\end{eqnarray*}

\begin{cor}\label{c:dim}
For any geometric point $(a,b)\in(\Ah\times\Bh)(\Omega)$, we have
\begin{equation}\label{eq:dim}
\dim_{\Omega}\calN_{a,b}=\sup_i\dim_{\Omega}\calM_{i,a,b}\leq\delta(a).
\end{equation} 
\end{cor}
\begin{proof}
The first equality follows from Lemma \ref{l:locisom}. Now we prove the inequality.

For each $x\in X(\Omega)$, $\calM^{x}_{i,a_x,b_x}$ is a subscheme of the affine Springer fiber of $\GL_n$ associated to a regular semisimple element with characteristic polynomial $a_x$. On the other hand, $\delta(a,x)$ is the dimension of that affine Springer fiber (cf. Bezrukavnikov's dimension formula \cite{Be}, and \cite[\S 3.6, 3.7]{NFL}). Therefore,
\begin{equation}\label{eq:locdim}
\dim_\Omega\calM^{x}_{i,a_x,b_x}\leq\delta(a,x).
\end{equation}
Now the inequality (\ref{eq:dim}) follows from the product formula (\ref{eq:pdGL}) (note that we only formulated the product formula for $k$-points $(a,b)$ of $\Ah\times\Bh$, but it has an obvious version for any geometric point of $\Ah\times\Bh$.)
\end{proof}

For each $\delta\geq0$, let $\calA^{\leq\delta}$ (resp. $\calA^{\geq\delta}$) be the open (resp. closed) subset of $\Ah$ consisting of those geometric points $a$ such that $\delta(a)\leq\delta$ (resp. $\delta(a)\geq\delta$). Recall the following estimate of Ng\^o (\cite[Proposition 5.4.2]{NFL}, based on local results of Goresky-Kottwitz-MacPherson on the root valuation strata \cite{GKM}). For each $\delta\geq0$, there is a number $c_\delta\geq0$, such that whenever $\deg(D)\geq c_\delta$, we have
\begin{equation}\label{eq:codim}
\codim_{\Ah_D}(\calA^{\geq\epsilon}_{D})\geq\epsilon,\hspace{1cm}\forall1\leq\epsilon\leq\delta.
\end{equation}
Here we write $\Ah_D$ to emphasize the dependence of $\Ah$ on $D$.

Finally, we prove the smallness.

\begin{prop}\label{p:small}
Fix $\delta\geq0$. For $\deg(D)\geq c_\delta, n(\deg(D_0)-g+1)>\delta+g_Y$, the morphisms
\begin{eqnarray*}
\fdel_{i}&:&\Mdel_i=\Mh_i|_{\Adel}\to\Adel\times\Bh\hspace{1cm}\forall -d\leq i\leq d\\
\gdel&:&\Ndel=\Nh|_{\Adel}\to\Adel\times\Bh
\end{eqnarray*}
are small.
\end{prop}
\begin{proof}
First, by Corollary \ref{c:globisom}, for any geometric point $(a,b)\in\Ah\times\Bh$ and any integer $i$, $\dim\calM_{i,a,b}\leq\dim\calN_{a,b}$, therefore the smallness of $\gdel$ implies the smallness of $\fdel_i$.

Now we prove that $\gdel$ is small. For each $j\geq1$, let $(\Ah\times\Bh)_{j}$ be the locus where the fiber of $\gh_i$ has dimension $j$. By Corollary \ref{c:dim}, we have
\begin{equation*}
(\Ah\times\Bh)_j\subset\calA^{\geq j}\times\Bh.
\end{equation*}
In particular, $(\Adel\times\Bh)_{j}$ is nonempty only if $j\leq\delta$.

\begin{claim}
For every geometric point $a\in\Ah(\Omega)$, the morphism $\gh_{a}:\Nh_a\to\Bh\otimes_k\Omega$ is dominant and generically finite when restricted to {\em every irreducible component} of $\Nh_a$.
\end{claim}
\begin{proof}
To simplify the notation, we base change everything to $\Spec(\Omega)$ via $a\in\Ah(\Omega)$ and omit $\Omega$ from the notations.

Let $\calB^a\subset\calB$ be the open subset consisting of geometric points $b$ such that $\gamma_{a,b}:\calO_{Y_a}(-D_0)\to\omega_{Y_a/X}(D_0)$ is an isomorphism at all singular points of $Y_a$. It is clear from the product formula (\ref{eq:pdU}) that if $b\in\calB^a$, then $\calN_{a,b}$ is finite and nonempty over $\Omega$. 

Note that by the Cartesian diagram (\ref{d:CartN}), the morphism $r:\Nh_a\to\NHit_a$ is smooth with fibers isomorphic to punctured vector spaces, therefore each irreducible component $C$ of $\Nh_a$ is of the form $r^{-1}(C')$ for some irreducible component $C'$ of $\NHit_a$. By \cite[Corollaire 4.14.3]{NFL}, $\NHit_a$ is equidimensional, therefore so is $\Nh_a$. Therefore, it suffices to show that for each irreducible component $C$ of $\Nh_a$, we have $\gh_{a}(C)\cap\calB^a\neq\varnothing$. Because then, $\gh_a|_{C}$ must be generically finite onto its image; but by the equidimensionality statement, we have $\dim C=\dim\Nh_a=\dim\calB^a=\dim\calB$, therefore $\gh_a(C)=\Bh$.

Now we fix an irreducible component $C\subset\Nh_a$ of the form $r^{-1}(C')$, where $C'\subset\NHit_a$ is an irreducible component. We argue that $\gh_{a}(C)\cap\calB^a\neq\varnothing$. By \cite[Proposition 4.14.1]{NFL}, the locus of $(\calF',h)\in C'$ where $\calF'$ is a line bundle on $Y_a'$ is dense. Let $(\calF',h)\in C'$ be such a point. Since $Y_a'$ is embedded in a smooth surface, hence Gorenstein, $\omega_{Y_a'/X}$ is a line bundle on $Y_a'$, hence $\calF'^\vee$ is also a line bundle.

Let $Z$ be the singular locus of $Y_a$. By the definition of $\delta(a)$, we have
\begin{equation*}
\#Z\leq2\dim H^0(Y_a,\calO_{\tilY_a}/\calO_{Y_a})=\delta(a).
\end{equation*}
where $\tilY_a$ is the normalization of $Y_a$.

To show that $\gh_a(C)\cap\calB^a\neq\varnothing$, we only have to find $\beta'\in\Hom(\calF',\omega_{Y_a'/X}(D_0))=H^0(Y_a',\calF'^\vee(D_0))$ such that $\coker(\beta')$ avoids the singular locus $Z'=\pi_a^{-1}(Z)$ of $Y_a'$, since the support of $\coker(\gamma)$ is the same as the projection of the support of $\coker(\beta')$. 

Consider the evaluation map
\begin{equation*}
0\to\calK\to\calF'^\vee(D_0)\xrightarrow{\oplus\ev(y')}\bigoplus_{y'\in Z'}\Omega(y')\to0.
\end{equation*}
Note that
\begin{equation*}
\chi(Y_a',\calK)=\chi(Y_a',\calF'^\vee(D_0))-\#Z'\geq2n(\deg(D_0)-g+1)-2\delta(a)\geq 2g_Y>g_Y'.
\end{equation*}
Therefore, by Grothendieck-Serre duality, $H^1(Y_a',\calK)=\Hom_{Y_a'}(\calK,\omega_{Y_a'})^\vee=0$. Hence the evaluation map
\begin{equation*}
H^0(Y_a',\calF'^\vee(D_0))\xrightarrow{\oplus\ev(y')}\bigoplus_{y'\in Z'}\Omega(y')
\end{equation*}
is surjective. In particular, we can find $\beta'\in H^0(Y_a',\calF'^\vee(D_0))$ which does not vanish at points in $Z'$. This proves the claim.
\end{proof}

Applying the above Claim to the geometric {\em generic} points of $\calA$ (note that $\Nh\to\Ah$ is surjective because $\NHit|_{\Ah}\to\Ah$ is), we see that $\gh$ restricted to every geometric irreducible component of $\Nh$ is generically finite and surjective.

Using the above Claim again, we see that for any geometric point $a\in\Ah$ and $j\geq1$, the locus of $b\in\Bh$ where $\dim\calN_{a,b}=j$ has codimension $\geq j+1$ in $\Bh$. Therefore,
\begin{equation*}
\codim_{\calA^{\geq j}\times\Bh}(\Ah\times\Bh)_{j}\geq j+1.
\end{equation*}
Since $\deg(D)\geq c_\delta$ and $j\leq\delta$, we have
\begin{eqnarray*}
&&\codim_{\Ah\times\Bh}(\Ah\times\Bh)_{j}\\
&\geq&\codim_{\Ah}(\calA^{\geq j})+\codim_{\calA^{\geq j}\times\Bh}(\Ah\times\Bh)_{j}\geq 2j+1.
\end{eqnarray*}
This proves the smallness.
\end{proof}

\section{Global formulation---matching of perverse sheaves}\label{s:match}

\subsection{A local system on $\Mh_i$}\label{ss:lsonM}
Consider the morphism
\begin{equation}\label{eq:nu}
\nu:\Mh_i\xrightarrow{r_\beta}\Quot^{d-i}(\omega(D_0)/Y/\Ah)\xrightarrow{\frN_{Y/\Ah}}\Sym^{d-i}(Y/\Ah)\to\Sym^{d-i}(X/k)
\end{equation}
where $\frN_{Y/\Ah}$ is the norm map defined by Grothendieck in \cite{FGA}. Here $\Sym^{d-i}(Y/\Ah)$ (resp. $\Sym^{d-i}(X/k)$) is the $(d-i)^{\textup{th}}$ symmetric power of $Y$ over $\Ah$ (resp. $X$ over $k$), constructed as the GIT quotient of the fibered power $(Y/\Ah)^{d-i}$ (resp. $(X/k)^{d-i}$) by the obvious action of the symmetric group $\Sigma_{d-i}$. The morphism $\Sym^{d-i}(Y/\Ah)\to\Sym^{d-i}(X/k)$ above is induced from the natural projection $(Y/\Ah)^{d-i}\to(X/k)^{d-i}$.

We will construct an \'etale double cover of $\Sym^{d-i}(X/k)$, which will give us a local system of rank one and order two on $\Sym^{d-i}(X/k)$, hence on $\Mh_i$.

The groups $(\ZZ/2)^{d-i}$ and $\Sigma_{d-i}$ both act on $(X'/k)^{d-i}$ and they together give an action of the semidirect product $\Gamma=(\ZZ/2)^{d-i}\rtimes\Sigma_{d-i}$ on $(X'/k)^{d-i}$. Then $\Sym^{d-i}(X/k)$ is the GIT quotient of $(X'/k)^{d-i}$ by $\Gamma$. We have a canonical surjective homomorphism $\epsilon:\Gamma\to\ZZ/2$ sending $(v_1,\cdots,v_{d-i},s)\mapsto v_1+\cdots+v_{d-i}$, where $v_j\in\ZZ/2$ and $s\in\Sigma_{d-i}$. Let $\Gamma_0$ be the kernel of $\epsilon$.

\begin{lemma}\label{l:Zetale}
Let $Z^{d-i}$ be the GIT quotient of $(X'/k)^{d-i}$ by $\Gamma_0$. Then the natural morphism $\zeta:Z^{d-i}\to\Sym^{d-i}(X/k)$ is an \'etale double cover. 
\end{lemma}
\begin{proof}
Let $t=x_1+\cdots+x_{d-i}$ be a geometric point of $\Sym^{d-i}(X)$ and let $t'=(x'_1,\cdots,x'_{d-i})$ be a geometric point of $(X'/k)^{d-i}$ over $t$. Then the point $t''=(\sigma(x'_1),x'_2,\cdots,x'_{d-i})$ is another geometric point of $(X'/k)^{d-i}$ which does not lie in the $\Gamma_0$-orbit of $t'$. Therefore $t'$ and $t''$ have different images $z'$ and $z''$ in $Z^{d-i}$. In other words, the reduced structure of $\zeta^{-1}(t)$ consists of (at least) two points $z',z''$. 

Consider the maps
\begin{equation*}
\xi:(X'/k)^{d-i}\xrightarrow{\eta}Z^{d-i}\xrightarrow{\zeta}\Sym^{d-i}(X/k)
\end{equation*}
where $\xi$ is finite flat of degree $2^{d-i}(d-i)!$. The degree of the geometric fibers $\eta^{-1}(z')$ and $\eta^{-1}(z'')$ are {\em at least} $2^{d-i-1}(d-i)!$ (because this is the generic degree). This forces the two degrees to be equal to $2^{d-i-1}(d-i)!$. This being true for any geometric point of $Z^{d-i}$, we conclude that the quotient map $(X'/k)^{d-i}\to Z^{d-i}$ is flat, hence faithfully flat. Therefore $\zeta:Z^{d-i}\to\Sym^{d-i}(X/k)$ is also flat (of degree 2). Since every fiber $\zeta^{-1}(x)$ already consists of two distinct points, these two points must be reduced. Therefore $\zeta$ is finite flat of degree 2 and unramified, hence an \'etale double cover.
\end{proof}

Let $L^{X}_{d-i}$ be the local system of rank one on $\Sym^{d-i}(X/k)$ associated to the \'etale double cover $Z^{d-i}\to\Sym^{d-i}(X/k)$ (cf. Notation \ref{n:loc}). We define
\begin{equation*}
L_{d-i}:=\nu^*L^X_{d-i}
\end{equation*}
to be the pull-back local system on $\Mh_i$.

We describe the stalks of the local system $L^X_{d-i}$ in more concrete terms. Let $L$ be the local system of rank one on $X$ associated to the \'etale double cover $\pi:X'\to X$. Let $C_0=\ker((\ZZ/2)^{d-i}\xrightarrow{\epsilon}\ZZ/2)$. Then $(X'/k)^{d-i}/C_0$ is an \'etale double cover of $(X/k)^{d-i}$.
\begin{lemma}\label{l:L}
\begin{enumerate}
\item []
\item We have a Cartesian diagram
\begin{equation*}
\xymatrix{(X'/k)^{d-i}/C_0\ar[r]\ar[d] & Z^{d-i}\ar[d]\\
(X/k)^{d-i}\ar[r]^{s_{d-i}} & \Sym^{d-i}(X/k)}
\end{equation*}
where the maps are all natural quotient maps.
\item $s_{d-i}^*L^X_{d-i}\cong L^{\boxtimes(d-i)}$.
\end{enumerate}
\begin{proof}
(1) follows from the fact the both vertical maps are \'etale (Lemma \ref{l:Zetale}).
(2) The local system of rank one associated to the \'etale double cover $(X'/k)^{d-i}/C_0\to(X/k)^{d-i}$ is clearly $L^{\boxtimes(d-i)}$. Therefore (2) follows from (1).
\end{proof}
\end{lemma}

\subsection{The incidence correspondence}
Recall that $\Asm$ is the open locus of $a\in\Ah$ where $Y_a$ is smooth (equivalently, the locus where $\delta(a)=0$). We assume that $\deg(D)$ is large enough ($\geq c_1$) so that $\Asm$ is nonempty, hence dense in $\Ah$. The norm maps $\frN_{Y'/\Ah}$ and $\frN_{Y/\Ah}$ are isomorphisms over $\Asm$. 

For each $-d\leq i\leq d$, consider the incidence correspondence
\begin{equation*}
I^{d-i,2d}\subset\Sym^{d-i}(Y/\Asm)\times\Sym^{2d}(Y/\Asm)
\end{equation*}
whose geometric fiber over $a\in\Asm$ classifies pairs of divisors $T\leq T'$ on $Y_a$ where $\deg(T)=d-i,\deg(T')=2d$. Let $\tau,\tau'$ be the projections of $I^{d-i,2d}$ to $\Sym^{d-i}(Y/\Asm)$ and $\Sym^{2d}(Y/\Asm)$.

\begin{lemma}\label{l:inc}
The incidence correspondence $I^{d-i,2d}$ is smooth over $\Asm$.
\end{lemma}
\begin{proof}
Since $Y\to\Asm$ is a smooth family of curves, we may identify the symmetric power $\Sym^{j}(Y/\Asm)$ with the Hilbert scheme $\Hilb^j(Y/\Asm)$. We apply the infinitesimal lifting criterion to prove the smoothness. Suppose $R_0$ is a local artinian ring and $R$ is a thickening of $R_0$. Let $a\in\Asm(R)$ with image $a_0\in\Asm(R_0)$. Let $T_0\subset T'_0\subset Y_{a_0}$ be subschemes of flat of degree $d-i$ and $2d$ over $R_0$. We want to find subschemes $T\subset T'\subset Y_a$, flat of degree $d-i$ and $2d$ over $R$, whose reduction to $R_0$ are precisely $T_0\subset T_0'$. We may assume that $T'_0$ is contained in an affine open subset $U\subset Y_a$. Let $U_0=U\cap Y_{a_0}$. Since $Y_{a_0}$ is a smooth family of curves over $R_0$, the subschemes $T_0$ and $T_0'$ are defined by the vanishing of functions $f_0$ and $f_0'\in\Gamma(U_0,\calO_{U_0})$ respectively. Since $T_0\subset T_0'$, we have $f'_0=f_0g_0$ for some function $g_0\in\Gamma(U_0,\calO_{U_0})$. Let $f,g$ be arbitrary liftings of $f_0,g_0$ to $\Gamma(U,\calO_U)$ (which exist because $U$ is affine), define $T$ and $T'$ to be the zero loci of $f$ and $fg$. Then it is easy to check that $T\subset T'\subset Y_a$ are flat over $R$ of the correct degree. This proves the smoothness of $I^{d-i,2d}$ over $\Asm$.
\end{proof}

We define a morphism
\begin{equation*}
\div:\Ah\times\Bh\xrightarrow{\coker{\gamma}}\Quot^{2d}(\omega(D_0)/Y/\Ah)\xrightarrow{\frN_{Y/\Ah}}\Sym^{2d}(Y/\Ah)
\end{equation*}
which sends $(a,b)$ to the cycle of $\coker(\gamma_{a,b})$ in $Y_a$.

\begin{lemma}\label{l:tCartM} We have a Cartesian diagram
\begin{equation}\label{d:tCartM}
\xymatrix{\Msm_i\ar[r]^{\widetilde{\div}}\ar[d]^{\fsm_i} & I^{d-i,2d}\ar[d]^{\tau'}\\
\Asm\times\Bh\ar[r]^{\div} & \Sym^{2d}(Y/\Asm)}
\end{equation}
Here the morphism $\widetilde{\div}:\Msm_i\to I^{d-i,2d}$ over a point $a\in\Asm$ sends $(\calF,\alpha,\beta)$ to the pair of divisors $\div(\beta)\subset\div(a,b)$ of $Y_a$.
\end{lemma}
\begin{proof}
We abbreviate $\Quot^{2d}(\omega(D_0)/Y/\Ah)$ by $\Quot^{2d}$. Let $\calQ$ be the universal quotient sheaf on $Y\times_{\Ah}\Quot^{2d}$. Then by the moduli interpretation given in Lemma \ref{l:modpic}, we have a Cartesian diagram
\begin{equation*}
\xymatrix{\Mh_i\ar[rr]^(.3){r_\beta}\ar[d]^{\fh_i} && \Quot^{d-i}\left(\calQ/(Y\times_{\Ah}\Quot^{2d})/\Quot^{2d}\right)\ar[d]\\
\Ah\times\Bh\ar[rr]^{\coker(\gamma)} && \Quot^{2d}} 
\end{equation*}
Restricting this diagram to $\Asm\times\Bh$, we have
\begin{eqnarray*} 
\Quot^{2d}|_{\Asm}&\cong&\Sym^{2d}(Y/\Asm)\\
\Quot^{d-i}\left(\calQ/(Y\times_{\Ah}\Quot^{2d})/\Quot^{2d}\right)|_{\Asm}&\cong& I^{d-i,2d}
\end{eqnarray*}
because the norm maps $\frN_{Y/\Asm}$ are isomorphisms. Therefore the diagram (\ref{d:tCartM}) is Cartesian.
\end{proof}

Let $L^Y_{d-i}$ be the pull-back of the local system $L^X_{d-i}$ via $\Sym^{d-i}(Y/\Asm)\to\Sym^{d-i}(X/k)$. Define
\begin{equation*}
K^Y_{d-i}:=\tau'_{*}\tau^*L^Y_{d-i}\in D^b_c(\Sym^{2d}(Y/\Asm),\Ql).
\end{equation*}

\begin{lemma}[Binomial expansion]\label{l:bin}
Let $\pi^Y_{2d}:\Sym^{2d}(Y'/\Asm)\to\Sym^{2d}(Y/\Asm)$ be the natural projection. Then there is a natural isomorphism
\begin{equation}\label{eq:bin}
\pi^{Y}_{2d,*}\Ql\cong\bigoplus_{i=-d}^{d}K^Y_{d-i}.
\end{equation}
\end{lemma}
\begin{proof}
Since the morphism $\tau':I^{d-i,2d}\to\Sym^{2d}(Y/\Asm)$ is finite and $I^{d-i,2d}$ is smooth over $k$ by Lemma \ref{l:inc}, we conclude that $K^Y_{d-i}$, being the direct image of a local system under $\tau'$, is a middle extension on $\Sym^{2d}(Y/\Asm)$. Similarly, $\pi^{Y}_{2d,*}\Ql$ is also a middle extension on $\Sym^{2d}(Y/\Asm)$. Therefore, to establish the isomorphism (\ref{eq:bin}), it suffices to establish such a natural isomorphism over a dense open subset of $\Sym^{2d}(Y/\Asm)$.

Now we consider the dense open subset $U\subset\Sym^{2d}(Y/\Asm)$ consisting of those divisors which are multiplicity-free. Let $U'$ (resp. $\tilU$, resp. $\tilU'$) be the preimage of $U$ in $\Sym^{2d}(Y'/\Asm)$ (resp. $(Y/\Asm)^{2d}$, resp. $(Y'/\Asm)^{2d}$). Then $s_{2d}:\tilU\to U$ is an \'etale Galois cover with Galois group $\Sigma_{2d}$, and we have a Cartesian diagram
\begin{equation*}
\xymatrix{\tilU'\ar[r]^{s'_{2d}}\ar[d]^{(\pi^Y)^{2d}} & U'\ar[d]^{\pi^Y_{2d}}\\
\tilU\ar[r]^{s_{2d}} & U}
\end{equation*}
Therefore we have
\begin{equation*}
(s_{2d}^*\pi^Y_{2d,*}\Ql)|_{\tilU}\cong((\pi^Y)^{2d}_*\Ql)|_{\tilU}=(\Ql\oplus L^Y)^{\boxtimes2d}|_{\tilU}.
\end{equation*}
where $L^Y$ is the pull-back of $L$ to $Y$.

On the other hand, we have a Cartesian diagram
\begin{equation*}
\xymatrix{\coprod_{J\subset\{1,2,\cdots,2d\},\#J=d-i}{\tilU_J}\ar[d]\ar[r] & I^{d-i,2d}|_{U}\ar[d]^{\tau'}\\
\tilU\ar[r]^{s_{2d}} & U}
\end{equation*}
where for each $J\subset\{1,2,\cdots,2d\}$, $\tilU_J\subset\Sym^{d-i}(Y/\Asm)\times\tilU$ is the graph of the morphism $\tau_J:(y_1,\cdots,y_{2d})\mapsto\sum_{j\in J}y_j$. Therefore we have
\begin{eqnarray*}
(s_{2d}^*K^Y_{d-i})|_{\tilU}\cong (s_{2d}^*\tau'_*\tau^*L^X_{d-i})|_{\tilU}\cong\bigoplus_{J\subset\{1,2,\cdots,2d\},\#J=d-i}\tau^*_JL^Y_{d-i}\\
\end{eqnarray*}
For each $J\subset\{1,2,\cdots,2d\}$ of cardinality $d-i$, let $p_J:\tilU\to(Y/\Asm)^J$ be the projection to those coordinates indexed by $J$. Then we have a factorization
\begin{equation*}
\tau_J:\tilU\xrightarrow{p_J}(Y/\Asm)^{J}\xrightarrow{s_J}\Sym^{d-i}(Y/\Asm).
\end{equation*}
By Lemma \ref{l:L}, we have $s_J^*L^Y_{d-i}\cong L^{Y,\boxtimes(d-i)}$. Therefore 
\begin{equation*}
\tau^*_JL^Y_{d-i}\cong p_J^*s_J^*L^Y_{d-i}\cong p_J^*((L^Y)^{\boxtimes(d-i)}).
\end{equation*}
Finally, we have an $\Sigma_{2d}$-equivariant isomorphism of local systems on $\tilU$:
\begin{eqnarray*}
\bigoplus_{i=-d}^{d}(s_{2d}^*K^Y_{d-i})|_{\tilU}\cong\bigoplus_{J\subset\{1,2,\cdots,2d\}}p_J^*((L^Y)^{\boxtimes\#J})\cong(\Ql\oplus L^Y)^{\boxtimes2d}|_{\tilU}.
\end{eqnarray*}
The last isomorphism justifies the nickname ``binomial expansion'' of this lemma. This $\Sigma_{2d}$-equivariant isomorphism descends to an isomorphism
\begin{equation*}
\bigoplus_{i=-d}^{d}K^Y_{d-i}|_{U}\cong(\pi^Y_{2d,*}\Ql)|_U
\end{equation*}
which proves the lemma.
\end{proof}

\subsection{A decomposition of $\gh_*\Ql$}
To state the next result, we need to define a technical notion.
\begin{defn}\label{def:pt} A commutative diagram of schemes
\begin{equation}\label{d:geomCart}
\xymatrix{X'\ar[r]^{\alpha}\ar[d]^{f'}& X\ar[d]^{f}\\
Y'\ar[r]^{\beta} & Y} 
\end{equation}
is said to be {\em pointwise Cartesian}, if for any algebraically closed field $\Omega$, the corresponding diagram of $\Omega$-points is Cartesian.
\end{defn}

For pointwise Cartesian diagrams with reasonable finiteness conditions, proper base change theorem also holds:

\begin{lemma}\label{l:bc}
Suppose we have a pointwise Cartesian diagram (\ref{d:geomCart}) where all maps are of finite type and $f,f'$ are proper. Let $F\in D^b_c(X,\Ql)$ be a constructible $\Ql$-complex on $X$, then we have a quasi-isomorphism
\begin{equation*}
\beta^*f_*F\cong f'_*\alpha^*F.
\end{equation*}
\end{lemma}
\begin{proof}
Let $X''=Y'\times_YX$ and $f'':X''\to Y'$ be the projection. Let $\xi:X'\to X''$ be the natural map over $Y'$, which is also proper. By the usual proper base change for Cartesian diagrams, we reduce to showing that for any $G\in D^b_c(X'',\Ql)$,
\begin{equation}\label{eq:G}
f''_!G\cong f'_!\xi^*G.
\end{equation}
But we have
\begin{equation*}
f'_*\xi^*G=f''_!\xi_*\xi^*G\cong f''_*(G\otimes\xi_*\Ql).
\end{equation*}
Therefore to show (\ref{eq:G}), it suffices to show that the natural map $\iota:\Ql\to\xi_*\Ql$ is a quasi-isomorphism. Since both $\xi_*\Ql$ is constructible, it suffices to show that $\iota$ is an isomorphism on the stalks of every geometry point $x''\in X''(\Omega)$, i.e.,
\begin{equation}\label{eq:pt}
\iota_{x''}:\Ql\to H^*(\xi^{-1}(x''),\Ql)
\end{equation}
is an isomorphism. By Definition \ref{def:pt}, $\xi^{-1}(x'')(\Omega)$ is a singleton. Therefore, the reduced structure of $\xi^{-1}(x'')$ is $\Spec\Omega$, and (\ref{eq:pt}) obviously holds.
\end{proof}

Consider the norm map
\begin{equation*}
\frN_{Y'/\Ah}:\Quot^{2d}(\omega(D_0)/Y'/\Ah)\to\Sym^{2d}(Y'/\Ah).
\end{equation*}

\begin{lemma}\label{l:tCartN}
The diagram
\begin{equation*}
\xymatrix{\Nsm\ar[d]^{\gsm}\ar[rr]^(.4){\frN_{Y'/\Asm}\circ r_\beta'} & & \Sym^{2d}(Y'/\Asm)\ar[d]^{\pi^Y_{2d}}\\
\Asm\times\Bh\ar[rr]^{\div} & & \Sym^{2d}(Y/\Asm)}
\end{equation*}
is pointwise Cartesian.
\end{lemma}
\begin{proof}
Since $\Nsm$ is reduced, to check the commutativity of the diagram, it suffices to check on geometric points. Therefore, we fix a geometric point $(a,b)\in\Asm(\Omega)\times\Bh(\Omega)$, and prove that the diagram is commutative and pointwise Cartesian at the same time. Again we omit $\Omega$ in the sequel. 

Let $\div(a,b)=\sum_{t=1}^rm_ty_t$ with $\{y_t\}$ distinct points on $Y_a$. Recall $\pi_a:Y_a'\to Y_a$ is the \'etale double cover induced from $\pi:X'\to X$. Let $\pi_a^{-1}(y_t)=\{y_t',y_t''\}$. Then a point $\calO_{Y_a'}(-D_0)\subset\calF'\subset\omega_{Y_a'/X}(D_0)$ is determined by the torsion sheaf $\calQ'=\omega_{Y_a'/X}(D_0)/\calF'$, which is a quotient of $\omega_{Y_a'/X}(D_0)/\calO_{Y_a'}(-D_0)$. Since $Y_a'$ is smooth, $\calQ'$ is in turn determined by its divisor $\sum_{t=1}^rm'_ty'_t+m_t''y_t''\leq\pi_a^{-1}(\div(a,b))$. The line bundle $\calF'$ is self-dual if and only if $m_t'+m_t''=m_t$ (cf. the proof of Lemma \ref{l:DVR}). Therefore, the image of the divisor of $\calQ$ is $\div(a,b)$. This proves the commutativity of the diagram. But this also shows that the map $\calN_{a,b}\to\pi^{Y,-1}_{2d}(\div(a,b))$ is a bijection on $\Omega$-points. This completes the proof.
\end{proof}

\begin{lemma} Fix $\delta\geq1,\deg(D)\geq c_\delta$ and $n(\deg(D_0)-g+1)\geq\delta+g_Y$. Then
\begin{enumerate}
\item For each $-d\leq i\leq d$, $\div^*K^Y_{d-i}[\dim\calA+\dim\calB]$ is a perverse sheaf.
\item We define
\begin{equation*}
K_{d-i}:=\jsm_{!*}(\div^*K^Y_{d-i}[\dim\calA+\dim\calB])[-\dim\calA-\dim\calB]
\end{equation*}
where $\jsm:\Asm\times\Bh\hookrightarrow\Adel\times\Bh$ is the open inclusion. Then we have a natural decomposition
\begin{equation*}
\gdel_*\Ql\cong\bigoplus_{i=-d}^{d}K_{d-i}.
\end{equation*}
\end{enumerate}
\end{lemma}
\begin{proof}
(1) By Lemma \ref{l:tCartN} and Lemma \ref{l:bc}, we have
\begin{equation*}
\div^*\pi^Y_{2d,*}\Ql\cong\gsm_*\Ql.
\end{equation*}
Since $\gsm$ is finite and $\Nsm$ is smooth, $\gsm_*\Ql[\dim\calA+\dim\calB]$ is a perverse sheaf. But by Lemma \ref{l:bin}, $K^Y_{d-i}$ is a direct summand of $\pi^Y_{2d,*}\Ql$, hence $\div^*K^Y_{d-i}$ is a direct summand of $\div^*\pi^Y_{2d,*}\Ql\cong\gsm_*\Ql$. Therefore $\div^*K^Y_{d-i}[\dim\calA+\dim\calB]$ is also a perverse sheaf.

(2) follows from the smallness of $\gdel$ proved in Proposition \ref{p:small}.
\end{proof}

\subsection{The global matching theorem}\label{ss:match} 
The global part of the main theorem of the paper is:
\begin{theorem}\label{th:match} Fix $\delta\geq1$,$\deg(D)\geq c_\delta$ and $n(\deg(D_0)-g+1)\geq\delta+g_Y$. Then for $-d\leq i\leq d$, there is a natural isomorphism in $D^b_c(\Adel\times\Bh)$:
\begin{equation}\label{eq:KL}
\fdel_{i,*}L_{d-i}\cong K_{d-i},
\end{equation}
hence an isomorphism
\begin{equation*}
\bigoplus_{i=-d}^{d}\fdel_{i,*}L_{d-i}\cong\gdel_*\Ql.
\end{equation*}
\end{theorem}

The proof of the theorem will occupy the rest of the section. We first prove a relatively easy case of the theorem.

\begin{lemma}\label{l:easymatch}
Theorem \ref{th:match} holds if $-d+2g_Y-1\leq i\leq d-2g_Y+1$.
\end{lemma}
\begin{proof}
For $\delta\geq1$,$\deg(D)\geq c_\delta$, $n(\deg(D_0)-g+1)>\delta+g_Y$ and $i$ in the above range, by Proposition \ref{p:smGL} and \ref{p:smU}, $\Mh_i$ and $\Nh$ are smooth. Moreover, by Proposition \ref{p:small}, the morphisms $\fdel_i$ are small. Therefore both $\fdel_{i,*}L_{d-i}$ and $K_{d-i}$ are middle extensions on $\Adel\times\Bh$. Hence it suffices to establish the isomorphism (\ref{eq:KL}) on the dense open subset $\Asm\times\Bh\subset\Adel\times\Bh$.

The morphism $\nu$ (see (\ref{eq:nu})) restricted on $\Msm_i$ factors as:
\begin{equation*}
\nu^{\textup{sm}}:\Msm_i\xrightarrow{\widetilde{\div}}I^{d-i,2d}\xrightarrow{\tau}\Sym^{d-i}(Y/\Asm)\to\Sym^{d-i}(X/k).
\end{equation*}
Therefore, we have
\begin{equation*}
L_{d-i}|_{\Msm_i}=\nu^{\textup{sm},*}L^X_{d-i}\cong\widetilde{\div}^*\tau^*L^Y_{d-i}. 
\end{equation*}
Applying proper base change to the Cartesian diagram (\ref{d:tCartM}), we get
\begin{eqnarray*}
\fsm_{i,*}L_{d-i}\cong\fsm_{i,*}\widetilde{\div}^*\tau^*L^Y_{d-i}&\cong&\div^*\tau'_*\tau^*L^Y_{d-i}\\
&\cong&\div^*K^Y_{d-i}=K_{d-i}|_{\Asm\times\Bh}.
\end{eqnarray*}
\end{proof}

Now we deal with the general case. Choose a closed point $\infty\in|X|$, let $\Ainf$ be the open subset of $\Ah$ where $p_a:Y_a\to X$ is \'etale over $\infty$; let $(\Ah\times\Bh)^{\infty}$ be the open subset of $\Ainf\times\Bh$ consisting of points $(a,b)$ such that $\gamma_{a,b}$ is an isomorphism over $\infty$ (or equivalently, $\div(a,b)$ avoids $\infty$). If we vary $\infty\in|X|$ (even if we are restricted to those $\infty$ which are split in $X'$), the open sets $(\Ah\times\Bh)^{\infty}$ obviously cover $\Ah\times\Bh$. Now the theorem reduces to the following Proposition. In fact, by this Proposition, for all $-d\leq i\leq d$, $\fdel_{i,*}L_{d-i}|_{(\Adel\times\Bh)^\infty}$ is then a middle extension on $(\Adel\times\Bh)^\infty$. Since this is true for any split $\infty$, we conclude that $\fdel_{i,*}L_{d-i}$ is a middle extension on $\Adel\times\Bh$. Then we can apply the argument of Lemma \ref{l:easymatch} to finish the proof of Theorem \ref{th:match}.

\begin{prop}
Suppose $\infty$ is split in $X'$. Then for each $-d\leq i\leq d$, the complex $\fdel_{i,*}L_{d-i}|_{(\Adel\times\Bh)^{\infty}}$ is a middle extension on $(\Adel\times\Bh)^{\infty}$.
\end{prop}
\begin{proof}
For each $N\geq0$, let $D_N=D_0+N\infty$ be a new divisor on $X$. We can use $D_N$ instead of $D_0$ to define the various moduli spaces and they enjoy all the properties we have proved so far. For those moduli spaces that depends on the choice of $D_N$, we add subscripts ``$D_0$'' or ``$D_N$'' to emphasize the dependence. Consider the embedding
\begin{equation*}
\iota_N:(\Ah\times\Bh_{D_0})^{\infty}\hookrightarrow\Ah\times\Bh_{D_0}\hookrightarrow\Ah\times\Bh_{D_N} 
\end{equation*}
given simply by the embeddings $H^0(X,\calO_X(2D_0)\otimes\calL(D)^{\otimes i})\hookrightarrow H^0(X,\calO_X(2D_N)\otimes\calL(D)^{\otimes i})$ induced by the natural embedding $\calO_X(2D_0)\hookrightarrow\calO_X(2D_N)$. Recall from Remark \ref{rm:defd} that we use $D$ and $D_0$ to define an integer $d=n(n-1)\deg(D)/2+n\deg(D_0)$. Let $d_N:=n(n-1)\deg(D)/2+n\deg(D_N)=d+Nn[k(\infty):k]$. 

Let $p^{-1}(\infty)$ is the preimage of $\infty$ in the universal spectral curve $Y^{\infty}\to\Ainf$, and $2Np^{-1}(\infty)$ is viewed as a flat family of subschemes of $Y^{\infty}\to\Ainf$, the $2N^{\textup{th}}$ infinitesimal neighborhood of $p^{-1}(\infty)$ in $Y$. For each $0\leq i_\infty\leq 2Nn[k(\infty):k]$, consider the Hilbert scheme of the scheme $2Np^{-1}(\infty)$ over $\Ainf$:
\begin{equation*}
h_{i_{\infty}}:\Hilb^{i_{\infty}}(2Np^{-1}(\infty))/\Ainf)\to\Ainf.
\end{equation*}
We claim that for any $-d_N\leq j\leq d_N$ we have a Cartesian diagram
\begin{equation}\label{d:DNM}
\xymatrix{\coprod_{}\Minf_{D_0,i}\times_{\Ainf}\Hilb^{i_{\infty}}(2Np^{-1}(\infty))/\Ainf)\ar[d]^{\finf_{D_0,i}\times h_{i_\infty}}\ar[r]^(.7){\tili_N} & \Mh_{D_N,j}\ar[d]^{\fh_{D_N,j}}\\
(\Ah\times\Bh_{D_0})^{\infty}\ar[r]^{\iota_N} & \Ah\times\Bh_{D_N}}
\end{equation}
where the disjoint union is over the set $i+(Nn-i_{\infty})[k(\infty):k]=j$ and $0\leq i_{\infty}\leq 2Nn$; the morphism $\finf_{D_0,i}:\Minf_{D_0,i}\to(\Ah\times\Bh_{D_0})^{\infty}$ is the restriction of $\finf_{D_0,i}$. In fact, the verification of this Cartesian diagram is same as proving the product formula \ref{p:pdGL}, except here we are working over the base $(\Ah\times\Bh_{D_0})^{\infty}$ rather than over a point. 

We also have a commutative diagram relating the maps $\nu$ (see the beginning of \S \ref{ss:lsonM}) in the situation of $D_0$ and $D_N$:
\begin{equation*}
\xymatrix{\Minf_{D_0,i}\times_{\Ainf}\Hilb^{i_{\infty}}(2Np^{-1}(\infty))/\Ainf)\ar[d]^{\nu_{D_0}}\ar[r]^(.7){\tili_N(i,i_\infty)} & \Mh_{D_N,j}\ar[d]^{\nu_{D_N}}\\
\Sym^{d-i}((X-\infty)/k)\ar[r]^{+i_{\infty}\infty} & \Sym^{d_N-j}(X/k)}
\end{equation*}
Since $\infty$ is split in $X'$, the pull-back of $L$ to $\Spec k(x)$ (via $x$) is a trivial local system. Hence by Lemma \ref{l:L}, we have $(+i_\infty\infty)^*L^X_{d_N-j}\cong L^X_{d-i}$. Therefore
\begin{equation*}
\tili_N(i,i_\infty)^*L_{d_N-j}=\tili_N(i,i_\infty)^*\nu_{D_N}^*L^X_{d_N-j}\cong \nu_{D_0}^*(+i_\infty\infty)^*L^X_{d-i}=L_{d-i}\boxtimes_{\Ainf}\Ql.
\end{equation*}
Applying proper base change to the diagram \ref{d:DNM}, we have
\begin{equation}\label{eq:bcM}
\iota_N^*\fh_{D_N,j,*}L_{d_N-j}\cong\bigoplus_{i+(Nn-i_{\infty})[k(\infty):k]=j}\finf_{D_0,i,*}L_{d-i}\boxtimes_{\Ainf} h_{i_\infty,*}\Ql.
\end{equation}

Since $\infty$ is split in $X'$, we fix a point $\infty'\in|X'|$ over $\infty$ and let $p'^{-1}(\infty')$ be its preimage in the family of curves $Y'^{\infty}\to\Ainf$. Similarly, we have a Hilbert scheme for $0\leq i_\infty\leq 2Nn[k(\infty):k]$:
\begin{equation*}
h'_{i_\infty}:\Hilb^{i_{\infty}}(2Np'^{-1}(\infty'))/\Ainf)\to\Ainf.
\end{equation*}
We also have a Cartesian diagram
\begin{equation}\label{d:DNN}
\xymatrix{\coprod_{i_\infty=0}^{2Nn}\Ninf_{D_0}\times_{\Ainf}\Hilb^{i_{\infty}}(2Np'^{-1}(\infty'))/\Ainf)\ar[d]^{\ginf_{D_0}\times h'_{i_\infty}}\ar[r]^(.7){\tili'_N} & \Nh_{D_N}\ar[d]^{\gh_{D_N}}\\
(\Ah\times\Bh_{D_0})^{\infty}\ar[r]^{\iota_N} & \Ah\times\Bh_{D_N}}
\end{equation}
where $\ginf_{D_0}:\Ninf_{D_0}\to(\Ah\times\Bh_{D_0})^{\infty}$ is the restriction of $\gh_{D_0}$. This diagram can be verified similarly using the argument of the product formula \ref{p:pdU}, and notice that since $\infty$ is split in $X'$, we have the isomorphism (\ref{eq:NM}) by Lemma \ref{l:locisom}, which also works over $(\Ah\times\Bh)^{\infty}$ to justify the appearance of $\Hilb^{i_{\infty}}(2Np'^{-1}(\infty'))/\Ainf)$ in the above diagram. Applying proper base change to the diagram \ref{d:DNN}, we get
\begin{equation}\label{eq:bcN}
\iota_N^*\gh_{D_N,*}\Ql\cong\bigoplus_{i_\infty=0}^{2Nn}\ginf_{D_0,*}\Ql\boxtimes_{\Ainf}h'_{i_\infty,*}\Ql.
\end{equation}

\begin{claim}
For each $0\leq i_\infty\leq 2Nn[k(\infty):k]$, $h_{i_\infty,*}\Ql$ and $h'_{i_\infty,*}\Ql$ are local systems on $\Ainf$.
\end{claim}
\begin{proof}
We prove the statement for $h_{i_\infty,*}\Ql$, and argument for the other one is the same. Let $S=2Np^{-1}(\infty)\to\Ainf$ be the universal family of subschemes of $Y|_{\Ainf}$. Then by the definition of $\Ainf$, $S^{\red}=p^{-1}(\infty)$ is finite \'etale over $\Ainf$. Let $T\to\Ainf$ be a Galois \'etale cover which splits $S^{\red}/\Ainf$. We base change the situation from $\Ainf$ to $T$:
\begin{equation*}
\xymatrix{S_T=\coprod_{r=1}^{n[k(\infty):k]} S_r\ar[r]\ar[d] & S\ar[d]\\T\ar[r]^{t} & \Ainf}
\end{equation*}
where $S_T$ splits into pieces $S_r$. Each $S_r^{\red}\cong T$ and $S_r$ is the $2N^{\textup{th}}$ infinitesimal neighborhood of $S^{\red}_r\subset Y\times_{\Ainf}T$. Therefore we have a Cartesian diagram
\begin{equation*}
\xymatrix{\prod_{r}\coprod_{i_r}\Hilb^{i_r}(S_r/T)\ar[r]\ar[d]_{\prod\coprod h^T_{i_r}} & \coprod_{i_\infty}\Hilb^{i_\infty}(S/\Ainf)\ar[d]^{h_{i_\infty}}\\T\ar[r]^{t} & \Ainf}
\end{equation*}
where the product is the fibered product over $T$. By proper base change, we have
\begin{equation}\label{eq:bcT}
\bigoplus_{i_\infty=0}^{2Nn[k(\infty):k]}t^*h_{i_\infty,*}\Ql\cong\bigotimes_{r=1}^{n[k(\infty):k]}\left(\bigoplus_{i_r=0}^{2N}h^T_{i_r,*}\Ql\right).
\end{equation}
Now each finite map $S_r\to T$ has geometric fibers of the form $\Spec\Omega[x]/x^{2N}$, therefore the reduced geometric fibers of each $\Hilb^{i_r}(S_r/T)$ is a single point if $0\leq i_r\leq 2N$, hence $h^T_{i_r,*}\Ql=\Ql$. By (\ref{eq:bcT}), we conclude that $t^*h_{i_\infty,*}\Ql$ is a local system on $T$, hence $h_{i_\infty,*}\Ql$ is a local system on $\Ainf$ because $t:T\to\Ainf$ is \'etale.
\end{proof}

Since each $\ginfdel_{D_0,*}\Ql$ is a middle extension (Proposition \ref{p:smU} and \ref{p:small}) and each $h'_{i_\infty,*}\Ql$ is a local system, $\iota_N^*\gdel_{D_N,*}\Ql$ is also a middle extension by the isomorphism (\ref{eq:bcN}). Recall that we have a decomposition $\gdel_{D_N,*}\Ql\cong\oplus_{j}K_{D_N,d_N-j}$, therefore each $\iota_N^*K_{D_N,d_N-j}$ is also a middle extension on $(\Adel\times\Bh_{D_0})^{\infty}$.

For $-d_N+2g_Y-1\leq j\leq d_N-2g_Y+1$, we know from Lemma \ref{l:easymatch} that
\begin{equation*}
\fdel_{D_N,j,*}L_{d_N-j}\cong K_{D_N,d_N-j}.
\end{equation*}
Hence
\begin{equation}\label{eq:bothmid}
\iota_N^*\fdel_{D_N,j,*}L_{d_N-j}\cong\iota_N^*K_{D_N,d_N-j}
\end{equation}
are both middle extensions on $(\Adel\times\Bh)^{\infty}$. Now we choose $N$ such that $Nn[k(\infty):k]\geq2g_Y-1$ (note that $g_Y$ only depends on $\deg(D)$, hence independent of $N$). Then for any $-d\leq i\leq d$, we have
\begin{equation*}
-d_N+2g_Y-1\leq i\leq d_N-2g_Y+1,
\end{equation*}
which means (according to the isomorphism (\ref{eq:bcM})) that $\finfdel_{D_0,i,*}L_{d-i}\boxtimes_{\Ainf}h_{Nn,*}\Ql$ appears as a direct summand of $\iota_N^*\fdel_{D_N,i,*}L_{d_N-i}$, which is a middle extension by (\ref{eq:bothmid}). Since $h_{Nn,*}\Ql$ is a local system by the above Claim, we conclude that $\finfdel_{D_0,i,*}L_{d-i}$ is also a middle extension on $(\Adel\times\Bh)^\infty$, for any $-d\leq i\leq d$, as desired.
\end{proof}

\section{Proof of the local main theorem}\label{s:pf}
We prove Theorem \ref{th:main} in this section. Suppose $\char(F)=\char(k)=p>\max\{n,2\}$. We are left with the case $k'/k$ nonsplit, and we assume this from here on. We fix a collection of invariants $(a^0,b^0)$ with $a^0_i,b^0_i\in\calO_E^{\sigma=(-1)^i}$ which is strongly regular semisimple. Suppose $\val_F(\Delta_{a_0,b_0})$ is even. Let $\delta(a^0)$ be the local Serre invariant associated to the algebra $R_{a^0}$. i.e., $\delta(a^0)=\dim_k(\tilR_{a_0}/R_{a^0})$ where $\tilR_{a^0}$ is the normalization of $R_{a^0}$. 

\subsection{Local constancy of the local moduli spaces}
In this subsection, we prove an analogous statement to \cite[Proposition 3.5.1]{NFL} in our situation. This is a geometric interpretation of the Harish-Chandra's theorem on local constancy of orbital integrals.

\begin{prop}\label{p:locconst}
There is an integer $N\geq1$ (depending on $(a^0,b^0)$) such that for any field $\Omega\supset k$ and any collection of invariants $(a,b)$ with $a_i,b_i\in(\calO_E\otimes_k\Omega)^{\sigma=(-1)^i}$, if $(a,b)\equiv(a^0,b^0)\mod \varpi^{N}$, then
\begin{enumerate}
\item $(a,b)$ is strongly regular semisimple;
\item $\delta(a)\leq\delta(a^0)+n/2$;
\item There are canonical isomorphisms of schemes over $\Omega$:
\begin{eqnarray*}
\Mloc_{i,a,b}\otimes_k\Omega&\cong&\Mloc_{i,a^0,b^0}\otimes_k\Omega\\
\Nloc_{a,b}\otimes_k\Omega&\cong&\Nloc_{a^0,b^0}\otimes_k\Omega
\end{eqnarray*}
\end{enumerate}
\end{prop}

\begin{proof}
We stick to the case $\Omega=k$, the general case is argued in the same way. First of all, by Lemma \ref{l:eqrs}, the strong regular semisimplicity of $(a,b)$ is checked by the nonvanishing of polynomials equations with $\calO_F$-coefficients in $a_i,b_i$: the discriminant $\Disc(P_{a})$ of the polynomial $P_{a}(t)=t^n-a_1t{n-1}\cdots+(-1)^na_n$ and the $\Delta$-invariant $\Delta_{a,b}$. Whenever $(a,b)\equiv(a^0,b^0) \mod\varpi^{N}$, we have
\begin{eqnarray*}
\Disc(P_{a})&\equiv&\Disc(P_{a^0}) \mod\varpi^{N}\\
\Delta_{a,b}&\equiv&\Delta_{a^0,b^0} \mod\varpi^{N}
\end{eqnarray*}
If we choose $N>\max\{\val_F(\Disc(P_{a^0})),\val_F(\Delta_{a^0,b^0})\}$, then whenever $(a,b)\equiv(a^0,b^0)\mod\varpi^{N}$, $\Disc(P_{a})$ and $\Delta_{a,b}$ are nonzero, hence $(a,b)$ are strongly regular semisimple.

Now fix this choice of $N$ and any $(a,b)$ such that $a_i,b_i\in(\calO_E\otimes_k\Omega)^{\sigma=(-1)^i}$ and $(a,b)\equiv(a^0,b^0)\mod\varpi^{N}$. Let $\gamma^0=\gamma_{a^0,b^0}$ and $\gamma=\gamma_{a,b}$

By the formula for $\delta(a)$ (\cite{Be} and \cite[\S 3.7]{NFL}), we have
\begin{equation*}
\delta(a)\leq\val_F(\Disc(P_a))/2=\val_F(\Disc(P_{a^0}))/2\leq\delta(a^0)+n/2.
\end{equation*}
 
Since $N\geq\val_F(\Delta_{a^0,b^0})=\val_F(\Delta_{a,b})$, we have
\begin{eqnarray}\label{eq:RmodN}
R_{a}^\vee/\gamma(R_{a})=(R_{a}^\vee/\varpi^NR_{a}^\vee)/\gamma(R_{a}/\varpi^NR_{a})\\
\label{eq:RpmodN}
R_{a^0}^\vee/\gamma^0(R_{a^0})=(R_{a^0}^\vee/\varpi^NR_{a^0}^\vee)/\gamma^0(R_{a^0}/\varpi^NR_{a^0}).
\end{eqnarray}

We prove that $\Mloc_{i,a,b}$ and $\Mloc_{i,a^0,b^0}$ are canonically isomorphic. Firstly we have a canonical isomorphism of $\calO_F/\varpi^N\calO_F$-algebras
\begin{equation*}
\iota:R_{a}/\varpi^NR_{a}\cong R_{a^0}/\varpi^NR_{a^0}. 
\end{equation*}
We also have a commutative diagram
\begin{equation*}
\xymatrix{R_{a}/\varpi^NR_{a}\ar[r]^{\gamma}\ar[d]^{\iota} & R_{a}^\vee/\varpi^NR_{a}^\vee\\
R_{a^0}/\varpi^NR_{a^0}\ar[r]^{\gamma^0} & R_{a^0}^\vee/\varpi^NR_{a^0}^\vee\ar[u]^{\iota^\vee}}
\end{equation*}
because $\gamma\mod\varpi^N$ only depends on $(a,b)\mod\varpi^N$. Therefore, from (\ref{eq:RmodN}) and (\ref{eq:RpmodN}) we conclude that $R_{a}^\vee/\gamma(R_{a})$ as an $R_{a}/\varpi^NR_{a}$-module is canonically isomorphic to $R_{a^0}^\vee/\gamma^0(R_{a^0})$ as an $R_{a^0}/\varpi^NR_{a^0}$-module. Looking back into the definition of $\Mloc_{i,a,b}$, we observe that this scheme canonically only depends on $R_{a}^\vee/\gamma(R_{a})$ as an $R_{a}$-module. Therefore we get a canonical isomorphism $\Mloc_{i,a,b}\cong\Mloc_{i,a^0,b^0}$.

The argument for the other isomorphism $\Nloc_{a,b}\cong\Nloc_{a^0,b^0}$ is the same.
\end{proof}

\subsection{Preparations}\label{ss:pre}
We fix a smooth, projective and geometrically connected curve $X$ over $k$ of genus $g$ with a $k$-point $x_0$. Also fix an \'etale double cover $\pi:X'\to X$, also geometrically connected, with only one closed point $x'_0$ above $x_0$. We choose identifications $\calO_F\isom\calO_{X,x_0}$ and $\calO_E\cong\calO_{X',x'_0}$.

Fix an integer $\delta\geq\delta(a^0)+n/2$. Fix effect divisors $D=2D'$ and $D_0$ on $X$, disjoint from $x_0$, such that
\begin{eqnarray*}
\deg(D)&\geq&\max\{c_\delta,2g+2Nn+1\}\\
\deg(D_0)&\geq&\frac{n-1}{2}\deg(D)+2g+\max\left\{\frac{\delta}{n},(1+\frac{1}{n})N\right\}\\
\end{eqnarray*}
These numerical assumptions will make sure that all the numerical conditions in the propositions or lemmas of the paper (including those which we are about to prove) are satisfied.

For each closed point $x:\Spec k(x)\to X$, let $\varpi_x$ be a uniformizing parameter of $\calO_{X,x}$ and $F_x$ the field of fractions of $\calO_{X,x}$; let $\Frob_x$ be the geometric Frobenius element in $\Gal(\overline{k(x)}/k(x))$. Let $E_x$ be the ring of total fractions of $\calO_{X',x}$. Let $\eta_x$ be the quadratic character of $F_x^\times$ associated to the quadratic extension $E_x/F_x$.

Recall that the double cover $\pi:X'\to X$ gives a local system $L$ according to Notation \ref{n:loc}.
Let $L_x=x^*L$ be the rank one local system on $\Spec k(x)$ given by the pull-back of $L$ via $x:\Spec k(x)\to X$. Then we have
\begin{equation}\label{eq:eta}
\eta_x(\varpi_x)=\Tr(\Frob_x,L_x)
\end{equation}
 
Let $\calA_0\times\calB_0\subset\calA\times\calB$ be the affine subspace consisting of $(a,b)$ such that
\begin{equation*}
(a,b)\equiv(a^0,b^0)\mod\varpi^N
\end{equation*}
 
\begin{lemma} Let $\calA'_0\subset\calA_0\cap\Adel$ the open locus of $a$ such that $Y_a$ is smooth away from $p_a^{-1}(x)$. If $\deg(D)\geq2g+2Nn+1$ and $2\deg(D_0)\geq2g+N-1$, then $\calA'_0$ and $\calB_0$ are nonempty.
\end{lemma}
\begin{proof}
First we have to make sure that $\calA_0$ and $\calB_0$ are nonempty. For this it suffices to show that the following evaluation maps at a $N^{\textup{th}}$ infinitesimal neighborhood of $x$ are surjective:
\begin{eqnarray*}
H^0(X,\calL(D)^{\otimes i})&\to&\calO_{F}/\varpi^N, 1\leq i\leq n\\
H^0(X,\calO_X(2D_0)\otimes\calL(D)^{\otimes i})&\to&\calO_F/\varpi^N, 0\leq i\leq n-1
\end{eqnarray*}
which is guaranteed as long as $\deg(D)\geq2g+N-1$ and $2\deg(D_0)\geq2g+N-1$.

Next we make sure that $\calA_0\cap\Ah$ is nonempty. By Proposition \ref{p:locconst}, any $a\in\calA_0$ is strongly regular semisimple at $x_0$, hence in particular $R_a$ is reduced. This implies $\calA_0\subset\Ared$. By Lemma \ref{l:intcomp},
\begin{equation*}
\codim_{\Ared}(\Ared-\Ah)\geq\deg(D)>2Nn\geq\codim_{\Ared}(\calA_0)
\end{equation*}
Therefore $\calA_0\cap\Ah\neq\varnothing$. 

Finally we prove that $\calA'_0$ is nonempty. We base change the whole situation to $\overline{k}$. We use the argument for the Bertini's theorem. For details, we refer to \cite[Proposition 4.6.1]{NFL}. We only point out the for $\deg(D)\geq2g+N+1$, the evaluation maps at both $x_0$ and any other $x\in X(\overline{k})$:
\begin{equation*}
\ev(x_0)\oplus\ev(x):H^0(\geom{X}{k},\calL(D)^{\otimes i})\to\calO_{F}/\varpi^N\bigoplus\calO_{\geom{X}{k},x}/\varpi_x^2
\end{equation*}
is also surjective. This is all we need to apply the Bertini argument.

We still have to check that for each $a\in\calA_0'$, $\delta(a)\leq\delta$. But since $Y_a$ is smooth away from $p_a^{-1}(x_0)$, we have $\delta(a)=\delta(a,x)\leq\delta(a^0)+n/2\leq\delta$ by Proposition \ref{p:locconst}(2).
\end{proof}

Since $\calA'_0$ is geometrically irreducible (because it is an open subset of an affine space) and nonempty, it contains a $k_m$-point for every $m\geq m_0$. Now we fix $m\geq m_0$ and fix a point $a\in\calA'_0(k_m)$. We base change the whole situation from $k$ to $k_m$; in particular, let $X_{m}=X\otimes_kk_m$. Since $D$ and $D_0$ are disjoint from $x_0$, the trivializations that we fixed allow us to get $\sigma$-equivariant isomorphisms:
\begin{eqnarray*}
R_{a}(\calO_E\otimes_kk_m)&\cong&\calO_{Y'_a}(-D_0)|_{\Spec\calO_{Y_a',x_0}};\\
R_{a}^\vee(\calO_E\otimes_kk_m)&\cong&\omega_{Y'_a/X}(D_0)|_{\Spec\calO_{Y_a',x_0}}.
\end{eqnarray*}

\begin{lemma}\label{l:nonvan} If $n(\deg(D_0)-g+1)\geq(n+1)N+g_Y$, then there exists $b\in\calB_0(k_m)$ such that for each closed point $x\neq x_0$ of $X_m$, $\calN^{x}_{a_x,b_x}(k_m(x))\neq\varnothing$.
\end{lemma}
\begin{proof}
We first choose any $b\in\calB_0(k_m)$ (which exists because $\calB_0\neq\varnothing$ is an affine space over $k$). The pair $(a,b)$ determines $\gamma=\gamma_{a,b}:R_a\otimes_kk_m\hookrightarrow R_a^\vee\otimes_kk_m$.

\begin{claim}
There exists $\calF'\in\cPic(Y_a')$ and a homomorphism $h:\calF'\hookrightarrow\sigma^*\calF'^\vee$ such that
\begin{enumerate}
\item $\sigma^*h^\vee=h$;
\item $h$ is an isomorphism away from $p'^{-1}_a(x_0')$;
\item There is an isomorphism $\iota:\calF'|_{\Spec\calO_{Y'_{a},x_0}}\cong R_a(\calO_E\otimes_kk_m)$ such that the following diagram is commutative
\begin{equation*}
\xymatrix{\calF'|_{\Spec\calO_{Y'_{a},x_0}}\ar[d]^{\iota}\ar[r]^{h} & \sigma^*\calF'^\vee|_{\Spec\calO_{Y'_{a},x_0}}\\
R_a(\calO_E\otimes_kk_m)\ar[r]^{\sigma_R\circ\gamma} & R_a^\vee(\calO_E\otimes_kk_m)\ar[u]^{\sigma^*\iota^\vee}}
\end{equation*}
\end{enumerate}
\end{claim}
\begin{proof}
Let $\tilR$ be the normalization of $R_a\otimes_kk_m$. We can write $\tilR=\prod_{t}^r\tilR_t$ where each $\tilR_t$ is a DVR with residue field $k_m(R_t)$. After choosing an isomorphism $R_a^\vee\otimes_kk_m\cong R_a\otimes_kk_m$ of $R_a\otimes_kk_m$-modules (which exists because $R_a\otimes_kk_m$ is Gorenstein), we can view $\gamma$ as an element of $R_a\otimes_kk_m$. In terms of the above decomposition, $\gamma=(\gamma_1,\cdots,\gamma_r)$ where $\gamma_t\in\tilR_t$, and we have
\begin{equation}\label{eq:valgamma}
\val_F(\Delta_{a^0,b^0})=\val_{F\otimes_kk_m}(\Delta_{a,b})=\sum_{t=1}^r\val_{R_t}(\gamma_t)[k_m(R_t):k_m].
\end{equation}
Since $\val_F(\Delta_{a^0,b^0})$ is even, the RHS is also even.

Then condition (1) is automatically satisfied once we have such an $h$ since $Y_a'$ is geometrically irreducible. Let $U'=Y'_a-p'^{-1}(x_0')$. Suppose $\calF_0'$ is a line bundle on $U'$ with an isomorphism $h_0:\calF_0'\isom\sigma^*\calF_0'^\vee$. Such a line bundle exists because we have a Kostant section $\Ah\to\NHit|_{\Ah}$ by \cite[\S 2.3]{LN} when $D=2D'$. The existence of the Kostant section requires that $\char(k)>n$. Any other $\sigma$-conjugate self-dual line bundle on $U'$ has the form $\calF'_0\otimes\calG$ where $\calG$ is a line bundle on $U'$ with an isomorphism $s:\calG\isom\sigma^*\calG^{-1}$. We want to glue $(\calF_0'\otimes\calG,h_0\otimes s)$ with $(R_a(\calO_E\otimes_kk_m),\sigma_R\circ\gamma)$ over the punctured formal neighborhood $U'_{x_0}$ of $p_a'^{-1}(x_0')$ to get the desired pair $(\calF',h)$. For this we have to choose an isomorphism
\begin{equation*}
\rho:\calF_0'\otimes\calG|_{U'_{x_0}}\cong R_a(E\otimes_kk_m)
\end{equation*}
such that the following diagram is commutative:
\begin{equation*}
\xymatrix{\calF_0'\otimes\calG|_{U'_{x_0}}\ar[r]^{h_0\otimes s}\ar[d]^{\rho} & \sigma^*(\calF_0'\otimes\calG)^\vee|_{U'_{x_0}}\ar@{=}[r] & \sigma^*\calF_0'^\vee\otimes\sigma^*\calG^{-1}|_{U'_{x_0}}\\
R_a(E\otimes_kk_m)\ar[r]^{\sigma\circ\gamma} & R_a(E\otimes_kk_m)\ar[u]^{\sigma^*\rho^\vee}}
\end{equation*}

We can translate this problem into the language of id\`eles. Let $K$ (resp. $K'$) be the function field of the geometrically connected curve $Y_a$ (resp. $Y_a'$) over $k_m$. Let $\AA_K$ (resp. $\AA_{K'}$) be the set of ad\`eles of $K$ (resp. $K'$) with $\OO_K$ (resp. $\OO_{K'}$) the product of local rings integers. Let $\Nm:\AA^\times_{K'}\to\AA^\times_K$ be the norm map. Then $\tilR$ is the product of local fields corresponding to places of $K$ over $x_0$. Thus we get a canonical embedding $\tilR\subset\AA_{K}$. In particular, we can identify $\gamma$ with an id\`ele in $\AA^\times_{K}$ which is nontrivial only at places over $x_0$.

A choice of the above pair $(\calG,s,\rho)$ (up to isomorphism) is the same as the choice of an id\`ele class $\theta\in K'^\times\backslash\AA^\times_{K'}/\OO^\times_{K'}$ such that $\Nm(\theta)=\theta\cdot\sigma\theta=\gamma$ as an id\`ele class in $K^\times\backslash\AA^\times_{K}/\OO^\times_{K}$. Let $W_{K'}$ and $W_K$ be the Weil groups of $K'$ and $K$. By class field theory, we have the following commutative diagram
\begin{equation*}
\xymatrix{K'^\times\backslash\AA^\times_{K'}\ar[d]^{\Art_{K'}}_{\wr}\ar[r]^{\Nm} & K^\times\backslash\AA^\times_{K}\ar[d]^{\Art_{K}}_{\wr}\ar[r]^{\partial} &  \ZZ/2\ar[d]^{\wr}\\
W^{\abel}_{K'}\ar[r] & W^{\abel}_K\ar[r] & \Gal(K'/K)}
\end{equation*}
where the map $\partial$ is defined by
\begin{equation*}
\partial(\xi)=\sum_{v \textup{ nonsplit}}\val_{K_v}(\xi_v)[k_v:k_m]\mod2
\end{equation*}
where $v$ runs over all places of $K$ which is nonsplit in $K'$ and $k_v$ is the residue field of $\calO_{K_v}$. Now to solve our problem, i.e., to solve the equation $\Nm(\theta)=\gamma$, the only obstruction is $\partial(\gamma)$. Since $x_0$ is nonsplit in $X'_m$, a place $v$ over $x_0$ is nonsplit in $K'$ if and only if $[k_v:k_m]$ is odd. Therefore
\begin{eqnarray*}
\partial(\gamma)&\equiv&\sum_{v|x_0, [k_v:k_m]\textup{ odd}}\val_{K_v}(\gamma_v)[k_v:k_m]\\
&\equiv&\sum_{v|x_0}\val_{K_v}(\gamma_v)[k_v:k_m]\\
&=&\sum_{t=1}^r\val_{R_t}(\gamma_t)[k(R_t):k_m]\mod 2.
\end{eqnarray*}
But the RHS is even by (\ref{eq:valgamma}), hence $\partial(\gamma)=0$. Therefore we can always find $\theta\in K'^\times\backslash\AA^\times_{K'}$ such that $\Nm(\theta)=\gamma$. Translating back into geometry, we have found the desired $(\calG,s,\rho)$ and hence we can glue $(\calF_0'\otimes\calG,h_0\otimes s)$ with $(R_a(\calO_E\otimes_kk_m),\sigma_R\circ\gamma)$ over $U'_{x_0}$ to get the desired $(\calF',h)$. The isomorphism $\iota:\calF'|_{\Spec\calO_{Y'_{a},x_0}}\cong R_a(\calO_E\otimes_kk_m)$ is tautologically given by the gluing.
\end{proof}

Now we pick such a triple $(\calF',h,\iota)$ from the above Claim. By construction, we have
\begin{equation*}
\chi(Y_a',\calF')=-2n(g-1)-\val_{F\otimes_kk_m}(\Delta_{a,b})/2
\end{equation*}

\begin{claim}
There is a homomorphism $\alpha':\calO_{Y_a'}(-D_0)\to\calF'$ such that the composition
\begin{equation*}
\iota\circ\alpha':R_a(\calO_E\otimes_kk_m)\cong\calO_{Y_a',x_0}\xrightarrow{\alpha'}\calF'|_{\Spec\calO_{Y'_{a},x_0}}\xrightarrow{\iota}R_a(\calO_E\otimes_kk_m)
\end{equation*}
is the identity modulo $\varpi^N$.
\end{claim}
\begin{proof}
Consider the following evaluation map at the $N^{\textup{th}}$ infinitesimal neighborhood of $p_a'^{-1}(x'_0)\subset Y_a'$:
\begin{equation*}
\ev:\calF'(D_0)\to\calF'(D_0)|_{\Spec\calO_{Y'_a,x_0}}\xrightarrow{\iota}R_a(\calO_E\otimes_kk_m)\otimes_{\calO_F}(\calO_F/\varpi^N).
\end{equation*}
Let $\calK$ be the kernel of $\ev$, which is a coherent sheaf on $Y_a'$. By Grothendieck-Serre duality, we have
\begin{equation*}
H^1(Y_a',\calK)\cong\Hom_{Y_a'}(\calK,\omega_{Y_a'})^\vee.
\end{equation*}
But since
\begin{eqnarray*}
\chi(Y_a',\calK)&=&\chi(Y_a',\calF'(D_0))-2nN\\
&=&-\val_{F\otimes_kk_m}(\Delta_{a,b})/2+2n(\deg(D_0)-g+1)-2nN\\
&\geq&2n(\deg(D_0)-g+1)-(2n+1)N\\
&\geq&2g_Y>g_Y'-1=\chi(Y_a',\omega_{Y_a'}),
\end{eqnarray*}
we must have $\Hom(\calK,\omega_{Y_a'})=0$, therefore $H^1(Y_a',\calK)=0$. This implies that
\begin{equation*}
\ev:\Hom_{Y_a'}(\calO_{Y_a'}(-D_0),\calF')=H^0(Y_a',\calF'(D_0))\to R_a(\calO_E\otimes_kk_m)\otimes_{\calO_F}(\calO_F/\varpi^N)
\end{equation*}
is surjective. Hence there exists $\alpha'\in\Hom_{Y_a'}(\calO_{Y_a'}(-D_0),\calF')$ such that $\iota\circ\ev(\alpha')\equiv1\mod\varpi^N$.
\end{proof}

Now let $\gamma'$ be the composition
\begin{equation*}
\calO_{Y_a'}(-D_0)\xrightarrow{\alpha'}\calF'\xrightarrow{h}\sigma^*\calF'^\vee\xrightarrow{\sigma^*\alpha'^\vee}\omega_{Y_a'}(D_0).
\end{equation*}
Then $\gamma'$ gives back another $b'\in\calB(k_m)$. From the construction it is clear that $b'\equiv b\mod\varpi^N$, therefore $b'\in\calB_0(k_m)$. Now for each closed point $x\neq x_0$ of $X_m$, the local moduli space $\calN^x_{a_x,b'_x}(k_m(x))$ is nonempty because it contains a point given by $\calF'|_{\Spec\calO_{Y_a',x}}$ (this is self-dual because $h$ is an isomorphism over $x\neq x_0$ by construction). Therefore, the pair $(a,b')$ satisfies the requirement of the Lemma.
\end{proof}

\subsection{The proof}
Now for each $m\geq m_0$, we have a pair $(a,b)\in\calA'_0(k_m)\times\calB_0(k_m)$ such that the condition in Lemma \ref{l:nonvan} holds. Using Theorem \ref{th:match}, taking the stalks of the two complexes in (\ref{eq:KL}) at the point $(a,b)$, we get an isomorphism of graded $\Frob^m_k$ modules:
\begin{equation}\label{eq:globalH}
\bigoplus_{i=-d}^{d}H^\bullet(\geom{\calM_{i,a,b}}{k_m},L_{d-i})\cong H^\bullet(\geom{\calN_{a,b}}{k_m},\Ql).
\end{equation}
By the product formulae (\ref{eq:pdGL}) and (\ref{eq:pdU}), we can rewrite (\ref{eq:globalH}) as
\begin{eqnarray*}
&&\bigotimes_{x\in|X_{m}|}\left(\bigoplus_{i_x}H^\bullet\left(\left(\Res_{k_m(x)/k_m}\calM^x_{i_x,a_x,b_x}\right)\otimes_{k_m}\overline{k},\Res_{k_m(x)/k_m}L_x^{\otimes i_x}\right)\right)\\
\label{eq:tensorR}
&\cong &\bigotimes_{x\in|X_{m}|}H^\bullet\left(\left(\Res_{k_m(x)/k_m}\calN^x_{a_x,b_x}\right)\otimes_{k_m}{\overline{k}},\Ql\right).
\end{eqnarray*}
Here $\Res_{k_m(x)/k_m}L_x$ is the local system of rank one on $\Res_{k_m(x)/k_m}\Spec k_m(x)$ induced from $L_x$. We have such a tensor product decomposition on the LHS because the local system $L_{d-i}$ on $\calM_{i,a,b}$, when pulled-back via the isomorphism (\ref{eq:pdGL}), becomes $\boxtimes_{x\in|X_{m}|}\Res_{k_m(x)/k_m}L_x^{\otimes i_x}$ on $\prod_{x\in|X_m|}\Res_{k_m(x)/k_m}\calM^{x}_{i_x,a_x,b_x}$ (cf. Lemma \ref{l:L}).

We use the following abbreviations
\begin{eqnarray*}
M^j_x&:=&\bigoplus_{i_x}H^j\left(\left(\Res_{k_m(x)/k_m}\calM^x_{i_x,a_x,b_x}\right)\otimes_{k_m}\overline{k},\Res_{k_m(x)/k_m}L_x^{\otimes i_x}\right)\\
N^j_x&:=&H^j(\left(\Res_{k_m(x)/k_m}\calN^x_{a_x,b_x}\right)\otimes_{k_m}\overline{k},\Ql).\\
M^j_0&:=&\bigoplus_{i=0}^{\val_F(\Delta_{a^0,b^0})}H^j\left(\geom{\Mloc_{i,a^0,b^0}}{k},\Ql(\eta_{k'/k})^{\otimes i}\right);\\
N^j_0&:=&H^j(\geom{\Nloc_{a^0,b^0}}{k},\Ql).
\end{eqnarray*}
where the first two are $\Frob_k^m$-modules and the last two are $\Frob_k$-modules.

By Proposition \ref{p:locconst}, we have
\begin{equation*}
M^j_{x_0}\cong M^j_0;\hspace{1cm}N^j_{x_0}\cong N^j_0
\end{equation*}
as $\Frob^m_k$-modules.

On the other hand, since we have assumed that $Y_a$ is smooth away from $p_a^{-1}(x_0)$, for any $x\neq x_0$, the local moduli spaces $\calM^x_{i_x,a_x,b_x}$ and $\calN^x_{a_x,b_x}$ are zero-dimensional, hence no higher cohomology. For $x\neq x_0$, we write $M_x$ (resp. $N_x$) for $M^0_x$ (resp. $N^0_x$). Therefore for each $j$, we get an isomorphism of $\Frob^m_k$-modules:
\begin{equation*}
M^j_0\otimes\left(\bigotimes_{x_0\neq x\in|X_m|}M_x\right)\cong N^j_0\otimes\left(\bigotimes_{x_0\neq x\in|X_m|} N_x\right).
\end{equation*}

Taking the traces of $\Frob^m_k$, and using the Lefschetz trace formula for $\calM^x_{i_x,a_x,b_x}$ and $\calN^{x}_{a_x,b_x}$, we get for any $j\geq0$,
\begin{eqnarray}\label{eq:jcount}
&&\Tr(\Frob^m_k,M^j_0)\prod_{x_0\neq x\in|X_m|}\left(\sum_{i_x}\eta_x(\varpi_x)^{i_x}\#\calM^{x}_{i_x,a_x,b_x}(k_m(x))\right)\\
&=&\Tr(\Frob^m_k,N^j_0)\prod_{x_0\neq x\in|X_m|}\#\calN^{x}_{a_x,b_x}(k_m(x)).
\end{eqnarray}
Here we used (\ref{eq:eta}). Since $Y_a$ is smooth away from $p_a^{-1}(x_0)$, for $x\neq x_0$, $R_{a_x}\cong\calO_{Y_a,x}$ is a product of DVRs, and we can apply Lemma \ref{l:DVR} to conclude that
\begin{equation*}
\sum_{i_x}\eta_x(\varpi_x)^{i_x}\#\calM^{x}_{i_x,a_x,b_x}(k_m(x))=\#\calN^{x}_{a_x,b_x}(k_m(x)),\forall x\neq x_0.
\end{equation*}
Moreover, the RHS for each $x\neq x_0$ is nonzero, because $(a,b)$ satisfies the condition in Lemma \ref{l:nonvan}. Therefore (\ref{eq:jcount}) implies
\begin{equation*}
\Tr(\Frob^m_k,M^j_0)=\Tr(\Frob^m_k,N^j_0),\forall j\geq0.
\end{equation*}

Since this is true for any $m\geq m_0$, we get an isomorphism of semi-simplified $\Frob_k$-modules
\begin{equation*}
M^{j,\textup{ss}}_0\cong N^{j,\textup{ss}}_0.
\end{equation*}
But by Lemma \ref{l:locisom}, $M^j_0$ and $N^j_0$ are isomorphic as $\Frob^2_k$-modules. Therefore we can conclude that $M^j_0\cong N^j_0$ as $\Frob_k$-modules since the unipotent part of the $\Frob_k$ action is uniquely determined by that of $\Frob_k^2$ by taking the square root. This proves the main theorem.

\end{document}